\documentclass{amsart}

 \usepackage{tikz}
\usetikzlibrary{snakes}

\usepackage{soul}

\newcommand{\ccball}{{\mathcal B}}
\newcommand{\ccsphere}{{\mathcal S}}
\renewcommand{\aa}{\hat{a}}
\renewcommand{\AA}{\hat{\sf a}}
\newcommand{\A}{{\sf a}}
\newcommand{\B}{{\sf b}}
\newcommand{\C}{{\sf c}}
\newcommand{\E}{{\sf e}}

\renewcommand{\S}{{\sf s}}
\newcommand{\T}{{\sf t}}
\newcommand{\V}{{\sf v}}

\newcommand{\U}{{\sf u}}
\newcommand{\W}{{\sf w}}
\newcommand{\X}{{\sf x}}

\newcommand{\M}{\mathfrak{m}}
\newcommand{\uglyH}{\mathfrak{h}}
\newcommand{\Z}{\mathbb{Z}}
\newcommand{\N}{\mathbb{N}}
\newcommand{\R}{\mathbb{R}}

\newcommand{\calL}{{\mathcal L}}
\newcommand{\III}{{\mathcal I}}

\newcommand{\Area}{{\sf Area}}
\newcommand{\Cay}{{\sf Cay}}
\newcommand{\std}{{\sf std}}
\DeclareMathOperator{\Cone}{{Cone}}
\DeclareMathOperator{\Dom}{{Dom}}
\DeclareMathOperator{\DomComp}{{DomComp}}
\DeclareMathOperator{\Shape}{{Shape}}
\DeclareMathOperator{\Patt}{{Patt}}
\newcommand{\SSS}{{\mathbb S}}
\newcommand{\BBB}{{\mathbb B}}
\newcommand{\sdot}{\! \cdot \!}
\newcommand{\sides}{{2k}}
\DeclareMathOperator{\CHull}{{CHull}}
\newcommand{\CC}{{\sc cc }}
\newcommand{\dCC}{d_{\hbox{\tiny \sc cc}}}

\newcommand{\normL}[1]{\| #1 \|
  \raisebox{-3pt}{\scriptsize{\ensuremath L}}}
\newcommand{\zero}{{\sf 0}}

\newcommand{\II}{{\hbox{\sc i}}}
\newcommand{\JJ}{{\hbox{\sc j}}}
\newcommand{\isoper}{{\mathbf{I}}}
\newcommand{\ii}{{\hbox{\tiny \sc i}}}
\newcommand{\jj}{{\hbox{\tiny \sc j}}}

\newcommand\eval[1]{\overline{#1}}

\let\hat\widehat

\theoremstyle{definition}
\newtheorem{Thm}{Theorem}
\newtheorem{theorem}[Thm]{Theorem}
\newtheorem{question}[Thm]{Question}

\newtheorem{proposition}[Thm]{Proposition}

\newtheorem{lemma}[Thm]{Lemma}
\newtheorem{corollary}[Thm]{Corollary}

\newtheorem{definition}[Thm]{Definition}

\newtheorem{example}[Thm]{Example}
\newtheorem*{remark}{Remark}

\begin{document}

\title{Rational growth in the Heisenberg group} \author{Moon Duchin
  and Michael Shapiro} 
\date{\today}  

\begin{abstract}
A group presentation is said to have rational growth if the
generating series associated to its growth function represents
 a rational function.  
A long-standing open question asks whether the Heisenberg group has {\em rational
growth for all finite generating sets}, and we settle this
question affirmatively.  
We also establish {\em almost-convexity for all finite generating sets}.
Previously, both of these properties were known to hold for all hyperbolic groups
and all virtually abelian groups, and there were no further examples in either case.
Our main method is a close description of the relationship between word metrics
and associated Carnot-Carath\'eodory Finsler metrics on the ambient Lie group.
We provide (non-regular) languages in any word metric that suffice to represent 
all group elements.
\end{abstract}

  \maketitle

\section{Introduction}

Growth functions of finitely-generated groups count the number of
elements that can be spelled as words in a generating alphabet, as a function of
spelling length.  Though the functions themselves depend on a choice
of generating set, they become group invariants under the standard
equivalence relation that allows affine rescaling of domain---in
particular, this preserves the property of having polynomial growth 
of a particular degree.

It has been known since the early 1970s that all nilpotent groups have
growth functions in the polynomial range, in fact bounded above and
below by polynomials of the same degree, and the degree was  computed by Bass
and Guivarc'h independently \cite{bass1972,guivarch1970,guivarch1973}.
A breakthrough theorem of Gromov states that in fact any group with
growth bounded above by a polynomial is virtually nilpotent
\cite{gromov-poly-growth}.

One can still wonder, however, whether the growth function is
precisely polynomial.  This turns out to be a bit too much to ask for
nilpotent groups.  Virtually abelian groups, for instance, have a slightly more
general property called {\em rational growth}: no matter what finite
generating set is chosen, the power series associated to the growth
function represents a rational function.  

Hyperbolic groups have rational growth for all generators---this is an
important theorem from the early 1980s for which credit can be shared
among Cannon, Thurston, and Gromov \cite{cannon1980, cannon1984,
  wordproc, gromov-hyp-gps}.  (This has an  interesting
history: Cannon's argument for fundamental groups of closed hyperbolic
manifolds directly generalized to hyperbolic groups once that
definition was in place.  And Thurston's definition of automatic
groups was partly motivated by these ideas.)  At almost the same time,
Benson established the same result for virtually abelian groups
\cite{benson1983}.  Given the work at the time understanding the
growth of nilpotent groups, it was a natural question to ask whether
nilpotent groups also have rational growth, which was open even for
the simplest non-abelian nilpotent group, the {\em integer Heisenberg
  group}.  This question was posed or referred to by many authors,
including \cite{dlh, grig-dlh, benson1987, shapiro1989, mann,
  stoll1996}.  By the late 1980s, Benson and Shapiro had independently
established one piece of this: the Heisenberg group has rational
growth in its standard generators.  We settle the full question here.
\begin{theorem}\label{thm::rationalGrowth}
The Heisenberg group has rational growth for all generating sets.
\end{theorem}

In the process of establishing this fact, we will get quite precise
information about the combinatorial geometry of Heisenberg geodesics
(Theorem~\ref{thm::shapes}) that will be useful in the further study
of the geometric group theory of $H(\Z)$, and should therefore
have applications to complex hyperbolic lattices with Heisenberg cusps.
We give remarks,
applications (including almost-convexity), and open questions in the last section.

\subsection{Literature}

We review what is known about rationality of growth in groups and classes of groups.

\medskip

\noindent \begin{tabular}{c|c|c}
For all $S$ & For at least one $S$ & For no $S$ \\
\hline
hyperbolic groups & some automatic groups & unsolvable word problem \\
virtually abelian groups & Coxeter groups, standard $S$ & intermediate growth \\
{\bf Heisenberg group} $H$ &  \st{$H$, standard $S$} & \\
& $H_5$, cubical $S$  & \\
& $BS(1,n)$, standard $S$ & 
\end{tabular}

\medskip

Automatic groups have rational growth when the automatic
  structure consists of geodesics.  In this case, there is a regular
  language of geodesics which bijects to the group; this is used in
\cite{neumann-shapiro} to study groups that act geometrically finitely
on hyperbolic space.  There are more examples belonging in the middle
category---known to have rational growth in a special generating
set---found in work of Barr\'e (quotients of triangular buildings),
Alonso (amalgams), Brazil (other Baumslag-Solitar groups), Johnson
(wreath products and torus knot groups), and others.  For references
and an excellent survey, see \cite{grig-dlh}.

The nilpotent cases go as follows.  As mentioned above,  \cite{benson1987, shapiro1989}
show that $H$ has rational growth in standard generators.  In
\cite{stoll1996}, Stoll proves the following remarkable result: the higher Heisenberg group
$H_5$
has transcendental growth in its standard generators, but rational
growth in a certain dual generating set, which we will call {\em
  cubical generators}.  (See Sec~\ref{sec::higherH} for a definition of $H_5$.)
On the other hand, Stoll establishes the following theorem to use as a
criterion for transcendental growth.

\begin{theorem}[Stoll \cite{stoll1996}]\label{thm::stoll}
If $\frac{\beta(n)}{\alpha\cdot n^d} \to 1$ and $\alpha$ is a
transcendental number, then $\BBB(x)=\sum\beta(n) x^n$ is a
transcendental function.
\end{theorem}

A volume computation gives $\alpha=\frac{6027+2\ln 2}{65610}$,
establishing that $(H_5,\std)$ has transcendental growth.  Over
fifteen years later, this (with small variations explained by Stoll)
still provides the only known example of a group with both rational
and irrational growth series.

\subsection*{Acknowledgments}

The authors have been thinking about this problem for a long time and
have many people to thank for interesting ideas and stimulating
conversations.  Particular thanks go to Christopher Mooney, to Cyril
Banderier, and to Laurent Bartholdi for initially communicating the
interest of this problem.  We thank Dylan Thurston for stimulating
discussions of computational aspects, including calculations of 
periods and coefficients of quasipolynomiality.

The first author is supported by NSF grants 
DMS-1207106 and DMS-1255442.

MS wishes to thank MD for inviting him on such a grand adventure.

\section{Outline}

In this paper, we will give a way to compare geodesics in the Cayley
graph of $H(\Z)$ with geodesics in a geometrically much simpler
continuous metric on $\R^3$.
We will show that word geodesics are not too different from these
simpler paths, classifying them by ``shape."  Since these geodesics in
$\R^3$ can be understood in terms of their projection into $\R^2$,
this allows us to use planar pictures to understand geodesics in
$H(\Z)$.

By an important theorem of Pansu \cite{pansu}, any
word metric on the Heisenberg group $H(\Z)$ is asymptotic to a
left-invariant metric on its ambient Lie group $H(\R)$, known as a
Carnot-Carath\'eodory ({\sc cc}) Finsler metric, which admits $\R^3$
coordinates.  (Pansu's theorem is much more general, and was further
generalized by Breuillard in \cite{breuillard}.)  Several authors have
studied the geodesics in these \CC metrics, including Krat, Stoll,
Breuillard, and Duchin--Mooney, and we will stay close to the notation
of \cite{dm}.  
It is to Pansu's \CC geodesics that we will 
compare our word geodesics.

The $xy$--plane in $\R^3$ will be identified with the $XY$--subspace
of the Lie algebra of $H(\R)$, and we will denote this plane by $\M$.
A generating set of $H$ induces a norm on $\M$ in a manner described
further below, and word geodesics come in two kinds.  The ``unstable"
kind behave much like geodesics in free abelian groups, and these are 
modeled by finitely many {\em patterns} which are close to reorderings
of free abelian geodesics.

The ``regular" geodesics behave differently.  Among these is a subset
whose projections to $\M$ fellow-travel the boundary of a
characteristic polygon $\isoper$ determined by $S$.  We will see that
this subset contains at least one geodesic for each group element.
Once we make this precise we will have described finitely many
languages, which we call {\em shapes}.  
The words of each language are parameterized by the lengths of runs of
particular letters.  Thus, each shape is a map from a subset of $\Z^M$
to spellings in $S^*$.

We will show that every group element has a geodesic spelling produced
by a pattern or a shape (Theorem~\ref{thm::shapes}), even though it is not true
that all geodesics are so obtained, nor is it necessarily true that
all spellings produced by these shapes  and patterns are geodesic.

The domains of patterns and shapes in $\Z^M$ are determined by linear
equalities, inequalities, and congruences, and so counting the spellings
enumerated by the shapes amounts to solving congruences in rational
polyhedra.  By a marvelous theorem of Benson \cite{benson1987},
enumeration over rational polyhedra yields a rational function.  

However, this is not yet sufficient for rationality of growth. 
For each group element $g=(a,b,c)$ we
must determine which shapes might produce a spelling for $g$ and among
these shapes we must determine which one(s) win the competition to produce a
shortest spelling.  While the horizontal position $(a,b)$ varies linearly over the shape's 
domain, the height or $z$--coordinate $c$ of the elements produced
by each shape $\omega$ varies {\em quadratically},
which poses a problem for counting.  However, we show that whenever two 
shapes compete for geodesity in spelling a group element, the {\em difference} in 
the heights they produce is linear on the domain of competition (Sec~\ref{sec::competition}).
These linear comparison lemmas then 
allow us to enumerate the elements of each 
length $n$ using only linear equations, inequalities, and congruences,
which finally establishes rational growth.

\subsection{Example:  Shapes in $\Z^2$}

To illustrate the idea of shapes of geodesics, consider the example of
$\Z^2$, first with standard generators $a,b$.  Here, we will introduce
four shapes: $a^{m}b^{-n}$, $a^{-m}b^{n}$, $a^{m}b^{n}$ and
$a^{-m}b^{-n}$.
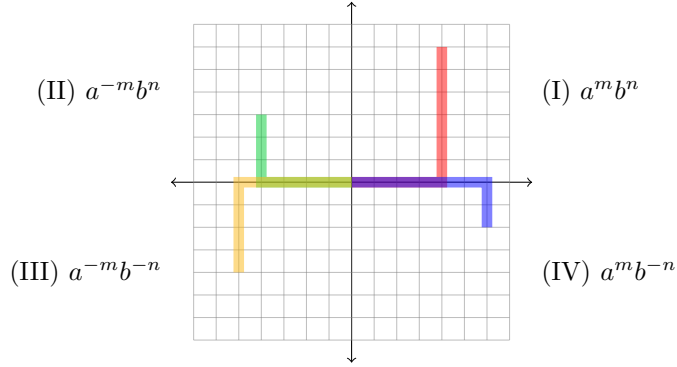
\begin{figure}[ht]
\begin{tikzpicture}[scale=.6]

\draw [<->] (-4,0)--(4,0) ; 
\draw [<->] (0,-4)--(0,4);
\foreach \x in {-3.5,-3,-2.5,...,3.5}
{\draw [gray,opacity=.5] (\x,-3.5)--(\x,3.5);
\draw [gray,opacity=.5] (-3.5,\x)--(3.5,\x);}

\draw [opacity=.5,red,line width=4] (0,0)--(2,0)--(2,3);
\draw [opacity=.5,blue, line width=4] (0,0)--(3,0)--(3,-1);
\draw [opacity=.5,green!80!blue,line width=4] (0,0)--(-2,0)--(-2,1.5);
\draw [opacity=.5,yellow!70!red,line width=4] (0,0)--(-2.5,0)--(-2.5,-2);
\node at (4,2) [right] {(I) $a^m b^n$};
\node at (4,-2) [right] {(IV) $a^m b^{-n}$};
\node at (-4,2) [left] {(II) $a^{-m} b^{n}$};
\node at (-4,-2) [left] {(III) $a^{-m} b^{-n}$};
\end{tikzpicture}
\caption{Four shapes of geodesics for $(\Z^2,\std)$.  (Take $m,n\ge 0$ in each case.)}
\end{figure}

One quickly observes a few basic properties:
\begin{itemize}
\item There are finitely many shapes.
\item Each shape is a language, and a map from a subset of some $\Z^M$
  to $S^*$.  (Here, $M=2$ for each shape, and the domain is the first
  quadrant of $\Z^2$.)
\item Every group element admits a geodesic spelling by at least one
  shape (even though not every geodesic is realized this way).
\end{itemize}

This case is too simple to capture some features of the situation, so
consider the slightly more complicated case of $\Z^2$ with {\em
  chess-knight} generators $\{(\pm 2,\pm 1),(\pm 1,\pm 2)\}$.
Consider the case of geodesically spelling the group element
$(100,100)$.  If we let $\U=(2,1)$, $\V=(1,2)$, and $\W=(2,-1)$ then
$\U^{33}\V^{33}$ reaches an adjacent position in $\Z^2$, but an exact
spelling (in fact a geodesic spelling) requires two more letters:
$(\W\V)\U^{33}\V^{33}$.  (This is because of the well-known property
of chess-knights that it takes several moves to arrive at an adjacent
square on the chessboard.)  These correction terms never have more
than three letters, so we can arrive at a finite list of shapes: every
shape has the form $\X\sdot\A_i^m \A_{i+1}^n$, where $\X$ is in the
ball of radius three, $\A_i$ and $\A_{i+1}$ are cyclically successive
generators, and $m,n\ge 0$.

Now we can add to the list of properties:

\begin{itemize}
\item A shape may not evaluate to a geodesic for every value of its
  arguments.
\item The set of positions reached by each shape is given by the
  simultaneous solution of finitely many linear inequalities and
  congruences in the plane.
\item If a group element is reached by more than one shape, there is a
  linear function that compares the spelling length required by each
  shape.
\end{itemize}

These are the essential features that we will establish in the
Heisenberg group for an appropriate finite list of shapes though the
linear comparison feature will play a somewhat different role.

\section{Background}

\subsection{Growth of groups}

Suppose a group $G$ is generated by the finite symmetric generating set
$S=S^{-1}$.  We take $S^n$ to be the set of all (unreduced) strings of
length $n$ in the elements of $S$ (sometimes called {\em spellings})
and $S^* = \cup_{n=0}^\infty S^n$ to be the set of all spellings of
any finite length.  This $S^*$ comes equipped with two important maps,
{\em spelling length} and {\em evaluation} into $G$.  Length, denoted
$\ell(\gamma)$, is defined on $\gamma \in S^n \subset S^*$ via
$\ell(\gamma) = n$.  Evaluation into $G$ is given by the monoid
homomorphism which carries concatenation in $S^*$ to group
multiplication in $G$.  
An element of $S^*$ can be thought
of as a path in the Cayley graph $\Cay(G,S)$ from $e$ to
the evaluation of $\gamma$.

We define the {\em word length} of a group element $g\in G$ by
$$|g| = |g|_S = \min\{\ell(\gamma) \mid \text{$\gamma \in S^*$ and
  $\eval{\gamma}=g$} \} ,$$ i.e., the shortest spelling length of any spelling.

The {\em sphere} and {\em ball} of radius $n$ are denoted $S_n,B_n$
respectively, and the associated growth functions are
\begin{align*}
\sigma(n):=\#S_n &= \#\{g \in G : |g|= n\} \ ; \\
\beta(n):=\#B_n &= \#\{g \in G : |g| \le n\} \ ,
\end{align*}
related of course by $\sigma(n)=\beta(n)-\beta(n-1)$.  Then we can
form associated generating functions, called the {\em spherical growth
  series} and the {\em growth series} of $(G,S)$, as follows:
$$\SSS(x) := \sum_{n=0}^\infty \sigma(n) x^n \ ; \qquad 
\BBB(x):= \sum_{n=0}^\infty \beta(n) x^n   .$$

Since $\sigma(n)\le \beta(n) \le \sum_{i=0}^n |S^i| =
\sum_{i=0}|S|^i$, the coefficients are bounded above by an
exponential, ensuring a positive radius of convergence for both
series.

We say that $(G,S)$ has {\em rational growth} if the growth series are
rational functions (i.e., each is a ratio of polynomials in $x$).
Note that the relationship between $\sigma$ and $\beta$ implies that
$(1-x)\BBB(x)=\SSS(x)$, so either is rational iff the other is.

It is a standard fact that rationality of a generating function
$F(x)=\sum f(n) x^n$ is equivalent to the property that the values
$f(n)$ satisfy a finite-depth linear recursion for $n\gg 1$, i.e.,
there exist $N_0$ and $P$ such that for $n>N_0$,
$$f(n+P)= a_0 \sdot f(n) + a_1\sdot f(n+1) + \cdots + a_{P-1}\sdot f(n+P-1).$$
(Here, the coefficients $a_i$ come from the same base field as the polynomials in the rational function.)

The growth of a regular language is necessarily rational with integer
coefficients, and therefore the values $\sigma(n)$ satisfy an integer
recursion.  In fact, this recursion can be described in terms of the
finite-state automaton which accepts the language, and therefore can
be written with non-negative integer coefficients in the recursion.
In the case of groups, if there is a generating set for which there is
a regular language of geodesics which bijects to the group, then the
corresponding growth function is rational.  This can be used to prove
rational growth for free abelian groups and for word hyperbolic
groups.

In this paper we focus on the integer Heisenberg group
$H=H(\Z)$ and consider its growth functions with various finite
generating sets.
Shapiro 1989 \cite{shapiro1989} shows that for the standard Heisenberg generators
$S=\std$, there is no regular language of geodesics for $(H,\std)$.
Nevertheless, the growth function is rational \cite{shapiro1989,benson1987}.  In this paper, we will
show the same holds for arbitrary generating sets.

\subsection{Rational families}\label{benson-machinery}

We now review material from Max Benson's papers
\cite{benson1983,benson1987},  articulating the principle that {\em
  counting in polyhedra is rational}. Benson uses these techniques in
\cite{benson1983} to show that virtually abelian groups have rational
growth with respect to arbitrary generating set. 

Suppose we have a parameter $n$ which we will take to lie in the
non-negative integers and we consider sets of points  $E(n)
\subset \Z^d$ defined by finitely many equalities, inequalities, and
congruences
$$\begin{cases}
\A_i \cdot \X = b_i(n)  \ ;\\
\A_j \cdot \X  \le b_j(n) \ ; \\
\A_k \cdot \X \equiv b_k(n)  \pmod {c_k} \ , 
\end{cases}$$
where each $\A_i$, $\A_j$ and $\A_k$ are in $\Z^d$, and each $b_i$,
$b_j$ and $b_k$ is an affine function of $n$ with integer
coefficients.  Such a sequence of sets $\{E(n)\}$ is called an {\em elementary family}.
Benson defines a {\em polyhedral family} $\{P(n)\}$ to be a finite union
of finite intersections of elementary families.  
If each $P(n)$ is bounded, then $\{P(n)\}$ is called a {\em bounded polyhedral family}.  

\begin{lemma}
The class of polyhedral families is closed under complementation, union, intersection, and 
set difference.  
\end{lemma}

\begin{proof}
This is clear for union, from the definition, and for intersection, by taking the combined
system of defining equalities, inequalities and congruences.

We now consider complementation.  
The complement of the solution set of an equation is the disjoint
union of the solution sets of two inequalities.  For an
inequality, the complement of its solution set is given by a single
inequality. For a congruence mod $r$,  the complement of its
solution set is the disjoint union of solutions to $r-1$
congruences. 

Finally, set difference can be built with intersection and complementation.
\end{proof}

\noindent Note also that the class of polyhedral families is closed under
affine push-forward; if $\{P(n)\}$ is a polyhedral family in $\Z^d$ and
$g:\Z^d \to \Z^m$ is an affine map, then  $\{g(P(n))\}$ is a polyhedral
family in $\Z^m$, and each is bounded if the other is.

\begin{theorem}[Counting over polyhedral families \cite{benson1983,benson1987}]\label{thm::rational-counting}
  Suppose that $\{P(n)\}$ is a bounded rational family in $\Z^d$
 and $ f:\Z^d \to \Z$ is a polynomial with integer
  coefficients.  Then
$$F(x) = \sum_{n=0}^\infty \quad \sum_{\V\in P(n)} f(\V) \, x^n $$
is a rational function of $x$.
\end{theorem}

\subsection{The Heisenberg groups}\label{sec::higherH}

Most of this paper will focus on the Heisenberg group $H(\Z)$, which
is also the first in the family $H_k$, $k=3,5,7,\dots$ of two-step
nilpotent groups realized as

$$
\begin{tikzpicture}[scale=1/3]

\draw (0,0) rectangle (10,10);
\foreach \x in {0,...,9}
{ \draw [gray] (\x,9-\x) rectangle  (\x+1,10-\x);}
\node at (.5,9.5) {$1$};
\node at (1.5,8.5) {$1$};
\filldraw (4.5,5.5) circle (0.02) (4.3,5.7) circle (0.02) (4.7,5.3) circle (0.02)  ;
\filldraw (4.5,9.5) circle (0.02) (4.3,9.5) circle (0.02) (4.7,9.5) circle (0.02)  ;
\filldraw (9.5,5.5) circle (0.02) (9.5,5.7) circle (0.02) (9.5,5.3) circle (0.02)  ;
\node at (9.5,.5) {$1$};
\node at (8.5,1.5) {$1$};
\node at (1.5,9.5) {$\Z$};
\node at (2.5,9.5) {$\Z$};
\node at (9.5,9.5) {$\Z$};
\node at (9.5,1.5) {$\Z$};
\node at (9.5,2.5) {$\Z$};

\draw [gray] (2,9)--(9,9) (9,2)--(9,9);
\node at (3,3) {$0$};
\node at (7,7) {$0$};
\end{tikzpicture}
$$
inside the $N\times N$ matrices, where
$N=\frac{k+3}2$.  (This
parametrization has $k$ as the number of integer parameters in each matrix.)  For
$i=1,2,\dots,N-2$, let $a_i$ be the $(1,i+1)$ elementary
matrix, let $b_i$ be the $(N,i+1)$ elementary matrix, 
and write $c$ for top-right elementary
matrix.  Then we have the commutator relations $[a_i,b_i]=c$ and all
other commutators are trivial.  Thus for any $k$, the commutator
subgroup is $\langle c\rangle$, so that the lower central series is
$$1 \trianglelefteq \Z \trianglelefteq H_k.$$ 

The well-known Bass-Guivarc'h formula for
the degree of polynomial growth in nilpotent groups tells us that 
the growth function of $H_k$ is $\beta(n)
\asymp n^d$ for $d=(k-1)\sdot 1 + 1\sdot 2=k+1$.

For the Heisenberg group $H(\Z)$, we will drop the subscripts and
write the elementary matrices as $\E_1,\E_2,\E_3$, so that
$[\E_1,\E_2]=\E_3$ and $[\E_1^m,\E_2^n]=\E_3^{mn}$.  The standard
generating set for $H(\Z)$ is $\{\E_1,\E_2\}^{\pm 1}$, and from the
above formula we know that the growth function in these generators is
bounded above and below by fourth-degree polynomials.

\subsection{Geometric model, spelling paths, and boost}

We will use the {\em exponential coordinates} on $H(\Z)\le H(\R)$ given by the following representation:
$$(a,b,c) \leftrightarrow \begin{pmatrix}
	1&a&c + \frac 12 ab  \\
	0 &1&b\\
	0&0&1
 \end{pmatrix} .$$
These coordinates have the property that $(a,b,c)^n=(na,nb,nc)$, and in 
this notation $\E_1=(1,0,0)$, $\E_2=(0,1,0)$, and $\E_3=(0,0,1)$.

For integers $a$ and $b$, define $\epsilon(a,b)$ to be $1/2$ if $a$
and $b$ are both odd, and $0$ otherwise.  In these coordinates,
$H(\Z)$ looks just like the standard lattice $\Z^3\subset \R^3$
shifted by $\epsilon$ in the $z$ direction, and the Haar measure on
$H(\R)$ is identified with Lebesgue measure in $\R^3$.

\begin{definition}
A {\em spelling path} is a string of letters from $S$, i.e., an
element of $S^*$, regarded as a path in the Cayley graph that
represents a group element from $H(\Z)$.  Define $\ell$, $(a,b)$, and
$z$ to be the {\em length}, {\em horizontal position}, and {\em
  height} of $\gamma$, respectively: if the group element represented
by $\gamma$ is $(a,b,c)$, then $\ell(\gamma)$ is the spelling length
of the string, $(a,b)\in \M$ is the projection of the endpoint, and
the {\em height} of the group element and hence the path is
$z(\gamma)=c$.  Also let the {\em shadow}, denoted $\pi(\gamma)$, be
the projection to $\M$ (the path in $\M$ obtained by concatenating the
projections of the generators to $\M$ in the order of appearance in
$\gamma$).

Define the area of a spelling path $\gamma$, denoted $z_A(\gamma)$, to
be the balayage area of its projection, that is, the signed area of
the concatenation of $\pi(\gamma)$ with the chord between its endpoint
and 0.  The {\em boost} of a generating letter $a_i$ is its height
$z(a_i)$.  Then the boost of a spelling, denoted $z_b(\gamma)$, is the
sum of the boosts of the letters in the spelling.
\end{definition}

Note that the height of a spelling path is equal to its balayage area plus its boost:
$z(\gamma)=z_b(\gamma)+z_A(\gamma)$.

\subsection{\CC~metrics and Pansu's theorem}

As mentioned above,  Pansu's theorem  states 
that the large-scale structure of the Cayley graph $(H,S)$  is a 
metric on $H(\R)$.  
It is not a Riemannian metric, but rather a
{\em sub-Finsler metric} called a \CC metric.  See \cite{breuillard,dm} for
some explicit descriptions of the geometry of the limit metric, and
\cite{capogna} for general background on sub-Riemannian geometry and
the Heisenberg group.  We collect a few salient features here.

The \CC metrics are defined as follows.
Let $\M$ denote the horizontal subspace of the Lie algebra $\uglyH$ of
$H(\R)$; that is, the span of the tangent vectors
$X=\left( \begin{smallmatrix} 0&1 &0 \\ 0 &0&0\\ 0&0&0
 \end{smallmatrix} \right)$
and 
$Y=\left( \begin{smallmatrix}
	0&0&0  \\
	0 &0&1\\
	0&0&0
 \end{smallmatrix} \right)$, and identify $\M$ with the $xy$--plane in $\R^3$ in exponential coordinates.  
 We can regard
$\M$ as a copy of $\R^2$ and make use of the linear projection $\pi:
H(\R) \to \M$ given by $(a,b,c)\mapsto (a,b)$.

 Fix a centrally symmetric convex polygon $L\subset \M$; this uniquely
 defines a norm ${\normL\cdot}$ on $\M$ for which $L$ is the unit
 sphere.  The push-forwards of $\M$ by left multiplication give {\em
   admissible planes} at every point in $H(\R)$, which are similarly
 normed; the plane field is a sub-bundle of the tangent bundle to
 $H(\R)$.  We say that a curve in $H(\R)$ is {\em admissible} if
it is piecewise differentiable and all of its tangent
 vectors lie in these normed planes.  The length of an admissible
 curve is simply the integral of the lengths of its tangent vectors,
 and it is easily verified that this is the same as the length in the
 $L$--norm of the projection $\pi(\gamma)$, and that any two points
 are connected by an admissible path.  Then the \CC distance
 $\dCC(x,y)$ is (well-)defined as the infimal length of an admissible
 path between $x$ and $y$.

In exponential coordinates, all \CC metrics are equipped with a
dilation $\delta_t(a,b,c)=(ta,tb,t^2c)$ that is a metric similarity,
scaling lengths and distances by $t$, areas in $\M$ by $t^2$, and
volumes by $t^4$.

Pansu also tells us which polygon $L$ is induced by a generating set
$S$: namely, $L$ is the boundary of the convex hull of the projection
$\pi(S)$ of the generators to $\M$.  For example, the two most basic
generating sets for $H(\Z)$ are $\{\E_1,\E_2\}^\pm$ and
$\{\E_1,\E_2,\E_3\}^\pm$.  In either case, the \CC metric is induced
by the $L^1$ norm on $\M$.  By contrast, if one took the nonstandard
generators $\{\E_1,\E_2,\E_1\E_2\}^\pm$, the polygon $L$ would be a
hexagon.

In this language, we can state this special case of Pansu's theorem as follows:  
for any finite symmetric generating set $S$ of $H(\Z)$,
$$\lim_{x\to\infty} \frac{\dCC(x,0)}{|x|_S}\to 1.$$ 

While Pansu's result extends to a statement for all nilpotent groups,
there is a substantial strengthening due to Krat \cite{krat} which was shown
only in the case of $H(\Z)$:  there is a global bound (depending on $S$)
in the additive difference between word and \CC lengths:
$\sup \Bigl| \dCC(x,0) - |x|_S \Bigr| <\infty$.   In Section~\ref{sec::bddDiff},
we will give a new proof of Krat's (and therefore Pansu's) result for $H(\Z)$.
We note that Breuillard (\cite{breuillard}) has shown that bounded difference
does {\em not} hold for all 2-step nilpotent groups, though on the other hand 
he has explained to us that 
arguments from \cite{breuillard-ledonne} can be adapted to show 
bounded difference for all of the higher Heisenberg groups.

\subsection{Significant directions, isoperimetrices, structure of \CC geodesics}\label{sec::CC}

It is a standard fact in Heisenberg geometry that for any admissible
path $\gamma$ based at the origin $0\in H(\R)$, the height or $z$
coordinate of $\gamma(t)$ is equal to its {\em balayage area}: the
signed (Lebesgue) area enclosed by the concatenation of the curve's
shadow $\pi(\gamma)$ with a straight line segment connecting its
endpoints.

As a consequence of the connection between height and balayage area,
we have a criterion for geodesity in the \CC metric: a curve $\gamma$
in $\M$ based at $(0,0)$ lifts to a geodesic in $H(\R)$ iff its
$L$-length is minimal among all curves with the same endpoints and
enclosing the same area.  As a result, to classify geodesics one uses the solution to the
isoperimetric problem in the normed plane $(\M,\normL\cdot)$.  By a
classical theorem of Busemann from 1947 \cite{busemann}, the
solution is described in terms of a polygon which he called the {\em
  isoperimetrix}.

\begin{definition}
For a finite symmetric generating set $S$, let $Q=\CHull(\pi(S))$ be
the convex hull of the projection of $S$ to $\M$ and let $L$ be its
boundary polygon, as above.  The {\em polar dual} of $Q$ is defined as
$Q^*=\{v\in\M : v\cdot x \le 1 \quad \forall x\in Q\}$ with respect to
the standard dot product.  Busemann's {\em isoperimetrix} is the polygon
$\partial(e^{i\pi/2}Q^*)$, obtained by rotating the polar dual of $Q$
through a right angle.  
\end{definition}

\begin{definition}
The vertices of the polygon $L$ will be labelled cyclically as
$\A_1,\ldots,\A_{\sides}$ and these vectors will be called {\em
  significant directions}.  
For the significant directions, we will extend the subscripts
periodically by defining $\A_{m}$ to equal $\A_{n}$ if $m\equiv n
\pmod{\sides }$.
  
  Each significant direction is the shadow
of at least one {\em significant generator} in $S$ and we will label
the generators projecting to $\A_i$ as $a_i,a_i',a_i''$, etc.
Elements of $S$ which project to the edges of $L$ are called {\em edge
  letters} and those that project properly inside $L$ are called
{\em interior letters}.
\end{definition}

\begin{remark} We will  maintain this font
distinction as much as possible to mark the difference between group elements $a\in H$ 
and their corresponding projections $\A\in\M$, the latter thought of 
as vectors in the plane.
\end{remark}

With this terminology, Busemann's theorem can be stated as follows.
Use Lebesgue measure on $\R^2$ for area and  length in
the Minkowski norm for perimeter.  Then up to dilation and translation
{\em the isoperimetrix is the unique closed curve realizing the
  maximal value of area divided by perimeter-squared}.  The following
properties follow from Busemann's construction.
\begin{itemize}
\item If the vertices of $Q$ have rational coordinates, then the same
  is true of the vertices of the isoperimetrix.
\item The edges of the isoperimetrix are parallel to the significant directions.
\end{itemize}
It follows that by clearing common denominators we can find positive 
integers 
$\sigma_1,\dots,\sigma_\sides$ with $\gcd=1$
and an integer $\lambda$
such that the edge vectors of  
$\lambda\partial(e^{i\pi/2}Q^*)$
are $\sigma_i \A_i$.

\begin{definition}
Define the {\em standard isoperimetrix}  to be the closed polygon 
 $\isoper=\isoper(S)$ having vertices 
 $$\zero, \quad  \sigma_1\A_1, \quad \sigma_1\A_1+\sigma_2\A_2, \quad  \dots$$
 (This is a translated and scaled copy of Busemann's curve.)
 \end{definition}

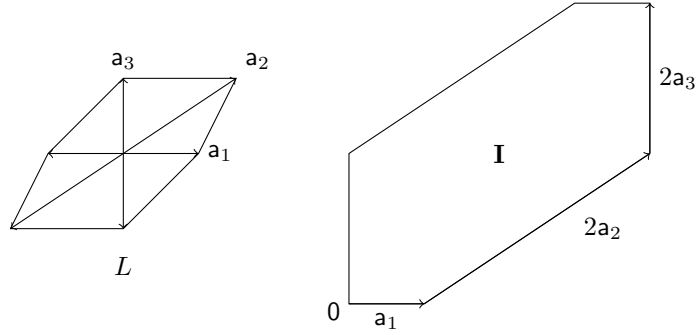
\begin{figure}[ht]
\begin{tikzpicture}
\draw (1,0)--(3/2,1)--(0,1)--(-1,0)--(-3/2,-1)--(0,-1)--cycle;
\draw [->] (0,0)-- (1,0) node [right] {$\A_1$};
\draw [->] (0,0)-- (3/2,1) node [above right] {$\A_2$};
\draw [->] (0,0)-- (0,1) node [above] {$\A_3$};
\draw [->] (0,0)-- (-1,0) ;
\draw [->] (0,0)-- (-3/2,-1);
\draw [->] (0,0)-- (0,-1) ;
\node at (0,-1.5) {$L$};
\begin{scope}[xshift=5cm,scale=2]
\node at (0,0) {$\isoper$};
\draw (1,1)--(1/2,1)--(-1,0)--(-1,-1)--(-1/2,-1)--(1,0)--cycle;
\draw [->] (-1,-1)--(-1/2,-1);
\node at (-1.1,-1.05)  {$\zero$};
\node at (-3/4,-1) [below] {$\A_1$};
\draw [->] (-1/2,-1)--(1,0);
\node at (1/2,-1/2) [right] {$2\A_2$};
\draw [->] (1,0)--(1,1);
\node at (1,1/2) [right] {$2\A_3$};
\end{scope}
\end{tikzpicture}
\caption{This example shows an isoperimetrix which is twice the rotated polar dual of the original polygon.
See \cite{busemann}.}
\end{figure}

\CC geodesics based at $\zero$ are classified in \cite{dm} into two
kinds: {\em regular geodesics}, which project to $\M$ as an arc of an
isoperimetrix, and {\em unstable geodesics}, which project to $\M$ as
geodesic in the $L$--norm.  

Fix  a polygon $L$ in $\M$, which determines a \CC metric on $H(\R)$.
Then for any $(a,b)\in \M$, 
 each length $\ell\ge \normL{(a,b)}$ uniquely determines a height $c=c(a,b,\ell)\ge 0$
so that there exists a regular geodesic  connecting $\zero$ to $(a,b,c)$ 
at length $\ell$.   That is, for each $\ell$ there exists a scale $s$ and 
a translation vector ${\sf q}$ so that 
$s\isoper+{\sf q}$  passes through $\zero$ and $(a,b)$; the subarc between
those two points has length $\ell$ with respect to $L$ and encloses area $c$, and it lifts 
to a \CC geodesic.  On the other hand if $\ell=\normL{(a,b)}$, there are 
$L$--norm geodesics connecting $\zero$ to $(a,b)$ with length $\ell$, and these can 
enclose any area in an interval of possibilities.

\begin{lemma}[Types of \CC geodesics \cite{dm}]\label{cc-types}
Let $c_0(a,b)=c(a,b,\normL{(a,b)})$ be the first height reached by a regular geodesic.
Then $(a,b,c_0)$ is reached by a geodesic that is both regular and unstable, 
and the same is true for $(a,b,-c_0)$.
Any point $(a,b,t)\in H(\R)$  is reached by only unstable 
geodesics if $|t|\le c_0$, and by only regular geodesics if $|t|\ge c_0$.

Furthermore, if $(a,b)$ is not in a significant direction then there are infinitely many 
unstable geodesics in the first case and a uniquely determined regular geodesic in the second.
\end{lemma}

Recall that ``stability of geodesics" means that for each pair of endpoints,
geodesics between the endpoints fellow travel with some fixed constant
(as in the Morse Lemma in hyperbolic geometry).  
Here, \CC geodesics of the second type
are called ``unstable" because they are highly non-unique and can fail 
to fellow travel arbitrarily badly.

\section{Euclidean geometry lemmas}\label{sec::Euc}

Below, we will use multiplicative vector notation for polygonal paths
in the plane, so that for instance $\A_1^t\A_2\A_1$ denotes the
concatenated path obtained by starting with the vector $t\A_1$
followed by the vector $\A_2$ followed by the vector $\A_1$, ending at
the point $(t+1)\A_1+\A_2$ in the plane.  Note that in this path
notation, the exponents need not be integers.

Any closed polygon in the plane with a vertex at the origin is traced
out by a path $P=\V_1\V_2\dots \V_r$ with $\sum \V_i=\zero$.  The
polygon is convex if and only if the vectors $\V_1,\dots,\V_r$ are
cyclically ordered (that is, if their arguments proceed in a monotone
fashion around the circle).  In this notation, the standard
isoperimetrix described above can be written $\isoper=\A_1^{\sigma_1}
\cdots \A_\sides^{\sigma_\sides}$.  For each $1\le i \le \sides$,
define $\aa_i=a_i^{\sigma_i}$ and $\AA_i=\A_i^{\sigma_i}$ so that
$\isoper=\AA_1 \cdots \AA_\sides$.  We refer to these as {\em blocks
  of significant letters} or more simply {\em significant blocks},
relying on context to distinguish between $\aa_i$ and $\AA_i$.
Note that we are treating $\aa_i$ as an element of $S^{\sigma_i}$, not a weighted
  generator added to the generating set $S$.

For any closed convex polygon $P$ in the plane, we consider the family
of polygons with the same ordered set of interior angles, i.e., the
family obtained by moving the sides of $P$ parallel to themselves.
Let us call this the {\em parallel family} of $P$. if $P=\V_1\V_2\dots
\V_r$, then an element of the parallel family is of the form
$P_\S=\V_1^{s_1}\V_2^{s_2} \dots \V_r^{s_r}$ for some
$\S=(s_1,\dots,s_r)\in \R^r$, and it is closed and convex if all
$s_i\ge 0$ and $\sum s_i\V_i=\zero$.

\begin{lemma}[Isoperimetric problem in parallel
    families]\label{lem::parallel-fam} 
Let $P=\V_1\V_2 \dots \V_r$ be a closed convex polygon in the plane.
For fixed arbitrary positive numbers
$\ell_1,\dots,\ell_r$ and for any $\lambda>0$, let
$$M(\lambda):=\{ \S\in [0,\infty)^r : \sum s_i\V_i=\zero, \quad
  \sum s_i\ell_i=\lambda \}.$$ Then the function $\Area(P_\S)$ has a
  unique local maximum on $M(\lambda)$.
\end{lemma}

\begin{proof}
$\Area(P_\S)$ varies quadratically over $\R^r$, and therefore also
  over the convex polytope $M(\lambda)$.  From the form of 
  $\Area(P_\S)$, one checks that it is a negative-definite quadratic form.
\end{proof}

This simple observation says that for an arbitrary convex polygon in
an arbitrary normed plane, the parallel family contains a unique set
of best proportions to maximize area relative to perimeter, by taking
the $\ell_i$ to be the lengths of the sides.  
(This is
slightly more general than what is implied by Busemann's theorem,
which for a given norm only treats the polygons in the parallel family of the
isoperimetrix.)

\begin{definition}
For a convex closed $P=\V_1\V_2 \dots \V_{r-1}\V_r$ as above, let the
indices be considered cyclically.  A $P$--{\em arc} $\tau$ of {\em
  scale} $s$ is a path $\V_\II^{s^-}\V_{\II+1}^s \dots
\V_{\JJ-1}^s\V_\JJ^{s^+}$ where $0\le s^-,s^+\le s$, and its {\em
  combinatorial length} is the sum of the exponents, 
  $\ell(\tau)=s^-+(\JJ-\II-1)s+s^+$.  Note
that $\tau$ begins at the origin and lies on a scaled and translated
copy of $P$, i.e., $\tau\subset sP+{\sf r}$.  
The {\em combinatorial type} of $\tau$
is the pair $(\II,\JJ)$ of starting and ending sides.  
There are two possible ambiguities:
first, if $\tau$ is one- or
two-sided, the scale is underdetermined, so we take $s$ to be the
maximum of $s^-$ and $s^+$.  Second, if $s^-$ or $s^+$ equals $0$
or $s$, then the arc is of more than one combinatorial type;
for instance, $\V_3^{100}\V_4^{100}\V_5^{100}$ is of types
$(2,5)$, $(2,6)$, $(3,5)$, and $(3,6)$.  

Given $K>0$, we
say that the arc {\em $K$--almost} has combinatorial type $(\II,\JJ)$
if it can be modified to an arc of combinatorial type $(\II,\JJ)$ by
modifying $s^+$ and $s^-$ by at most $K$ (possibly making them equal 
either $0$ or $s$ to change type).  
If there is $(\II,\JJ)$ so that $\tau$ and $\tau'$ both $K$--almost have
combinatorial type $(\II,\JJ)$, we say that they {\em
  $K$--almost have the same combinatorial type}.
  
In the special case of a $P$--arc $\tau=\V_\II^{s^-}\V_{\II+1}^s \dots
\V_{\II-1}^s\V_{\II}^{s^+}$ of type $(\II,\II)$, 
a {\em weight-shifted arc} is 
any $\hat \tau = \V_\II^{t^-}\V_{\II+1}^s \dots
\V_{\II-1}^s\V_{\II}^{t^+}$ where $t^-+t^+=s^-+s^+$.  Note that the family of weight-shifted
arcs for $\tau$ all reach the same endpoint and enclose the same area.
In the special case of a $P$--arc $\tau=\V_\II^{s^-}\V_{\II+1}^s \dots
\V_{\II-2}^s\V_{\II-1}^{s^+}$ of type $(\II,\II-1)$, 
a {\em cyclic permutation} of $\tau$ is $\bar\tau=\V_\JJ^{s}\V_{\JJ+1}^s \dots 
\V_{\II-2}^s\V_{\II-1}^{s^+}\V_\II^{s^-}\V_{\II+1}^s\dots 
\V_{\JJ-1}^s$ for any $\JJ$.  
Note that each cyclic permutation of $\tau$ reaches
the same endpoint.
  \end{definition}

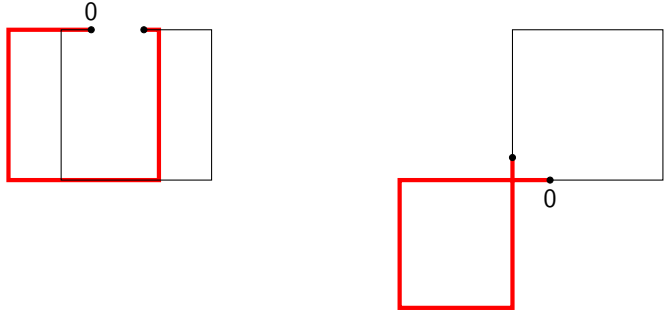
\begin{figure}[ht]
\begin{tikzpicture}[scale=2]
\draw [red, ultra thick] (0.2,1)--(-.35,1)--(-.35,0)--(.65,0)--(.65,1)--(.55,1);
\draw (0.2,1)--(0,1)--(0,0)--(1,0)--(1,1)--(.55,1);
\filldraw (0.2,1) circle (.02) node [above] {$\zero$};
\filldraw (0.55,1) circle (.02);

\begin{scope}[xshift=3cm]
\draw [red, ultra thick] (.25,0)--(-.75,0)--(-.75,-.85)--(0,-.85)--(0,.15);
\draw (.25,0)--(1,0)--(1,1)--(0,1)--(0,.15);
\filldraw (0.25,0) circle (.02) node [below] {$\zero$};
\filldraw (0,.15) circle (.02);
\end{scope}

\end{tikzpicture}
\caption{Shifting weight (left) and cyclically permuting (right).}
\end{figure}

\begin{lemma}[Combinatorial types of nearby arcs]\label{lem::comb-types}
Fix $K_1$ and $K_2$.  Then there are $K_3$, $K_4$ with the
following property.  If $\tau$ and $\tau'$ are $P$--arcs (based at the origin) whose
combinatorial lengths are within $K_1$ 
and whose endpoints are
within distance $K_2$, then their scales differ by at most $K_3$.
Further, after possible weight-shifting and cyclic permutation, they $K_4$--fellow travel.
\end{lemma}

\begin{proof}
Suppose $\tau$ is a $P$--arc with endpoint $(a,b)\in \R^2$, and the
length of $\tau$ is $\ell$.  By convexity of $P$, there are only very
limited ways to find $(a,b)$ as a chord of a scaled copy of $P$ with
given arclength, and this restricts the shape of $\tau$.

If $(a,b)=(0,0)$ then $\tau$ can be any translate of $P$ 
containing the origin.  Thus for endpoints near $(0,0)$, the arc can be nearly of type $(\II,\II)$ for 
any $\II$, but the scale is determined by the length.  
If one is of nearly of type $(\II,\II)$ and the other is of type $(\JJ,\JJ)$, then weight-shifting
followed by  cyclic permutation
suffices to make them fellow-travelers.

If $(a,b)$ is a nonzero multiple of 
$\V_\II$ and $\ell$ is sufficiently long, then the $P$--arc must be of type $(\II,\II)$.
In this case there is clearly a
family of polygons with the same $(a,b,\ell)$ and type $(\II,\II)$
obtained by shifting weight between $s^-$ and $s^+$, and these are the
only solutions to the chord problem.  So if one of the arcs, say $\tau$, has an endpoint precisely
on the $\V_\II$ direction, then it admits a weight-shifted family of $P$--arcs and we might 
have to choose the right one to match $\tau'$.  

The only remaining case is that $(a,b)$ is a nonzero vector which is not parallel to any $\V_i$, 
in which case the
triple $(a,b,\ell)$ uniquely determines not only $s$ but determines
$\tau$ completely (by convexity of $P$), and the starting side and ending side 
are different ($\II\neq\JJ$).
Within a combinatorial type, the scale $s$ is a
linear function of $(a,b,\ell)$, and indeed it is piecewise linear
(and continuous) across combinatorial types as $(a,b)$ varies over
the sector between any successive $\V_i,\V_{i+1}$.  
(See \cite{dm} for details and examples.)  Being far from the origin forces  arcs 
with nearby endpoints to be nearly of the same combinatorial type,
and so they fellow travel.  
\end{proof}

$P$--arcs have the proportions required to belong to $P$ as sub-arcs.
Generalizing slightly, in the family 
$$\gamma_\T = \C_{\II-1} \sdot \V_\II^{t_\ii} \sdot \C_{\II} \sdot \V_{\II+1}^{t{\ii+1}}\dots
\C_{\JJ-1} \sdot \V_{\JJ}^{t_\jj}\sdot \C_{\JJ}$$ 
as $\T \in\R^r$ varies,
a path with $\JJ-\II\ge 2$ will be called {\em balanced} if
 $t_\II,t_\JJ\le t_{\II+1}=t_{\II+2}=\dots=t_{\JJ-1}$.
 Such a path is {\em $K$--almost balanced} if all of these equalities
 and inequalities hold within $K$, i.e., 
 $|t_i-t_j|\le K$ for $\II<i,j<\JJ$ and $t_\II,t_\JJ \le t_i+K$ for all $\II<i<\JJ$.

Below, we will take the area of a not-necessarily-closed  path
to be its balayage area:  the signed area enclosed by concatenating the path 
with the chord from its endpoint to its start point.

\begin{lemma}[Balancing paths]\label{lem::balancing}
Suppose a closed convex polygon $P=\V_1\V_2 \dots \V_r$ encloses
maximal area among all closed $P_\S$ with $\sum_{i=1}^r s_i=r$.  Let
$$\gamma_\T = \C_{\II-1} \sdot \V_\II^{t_\ii} \sdot \C_{\II} \sdot \V_{\II+1}^{t{\ii+1}}\dots
\C_{\JJ-1} \sdot \V_{\JJ}^{t_\jj}\sdot \C_{\JJ},$$ 
and consider the affine subspace 
$$M(\lambda):=\{ {{\sf t}}\in \R^r : \sum_{i=\II}^\JJ t_i=\lambda
\}.$$ Then for every $\Delta>0$ and every vector $(k_1,\ldots,k_r)\in
\R^r$, there is a constant $K$ such that, for $\lambda\gg \Delta$ and
any translated lattice $L\le M(\lambda)$ with
diameter of its fundamental domain $\le \Delta$, the
maximum value of $f(\gamma_\T) = \Area(\gamma_\T)+\sum_{i=\II}^\JJ
k_it_i$ over $\T\in L$ occurs when the path is $K$--almost balanced
with respect to $P$.
\end{lemma}

\begin{proof}
Given $\gamma_\T$, we consider a path $\tau$ of the form  $\V_\II^* \V_{\II+1}^* \dots
\V_\JJ^*$ whose beginning and end differ from that of $\gamma_\T$ by
``straightening'' the inital and final corners of $\gamma_\T$.  In
this way, the area of $\tau$ differs from that of $\gamma_\T$ by a
constant which is independent of both $\lambda$ and $\T$.

Consider the balayage area $\Area(\tau)$.  This is given by the sums
of determinants whose entries are affine in $\T$ It follows that this
is quadratic in $\T$ and is negative definite on $M(\lambda)$.  It is
not hard to see that the maxima for $\Area(\tau)$ lie on a ray
emanating from the origin and that the eigenvalues of the negative
definite form vary inversely with $\lambda$.  This follows from the
fact that area grows quadratically with respect to the scale.  Thus,
the level curves in each $M(\lambda)$ are identical, though the values
along each of these curves differ with $\lambda$.  Clearly
$\Area(\tau)$  restricted to $L$ is maximized on the innermost lattice
point with respect to these curves and this is within $\Delta$ of the
real maximum for $\Area(\tau)$. 

Completing the square with respect to the $k_i$, we see that
$f(\gamma_\T)$ is negative definite on $M(\lambda)$ with the same
eigenvalues as $\Area(\gamma_\T)$ but with its maximum on $M(\lambda)$
shifted by a fixed amount which independent of $\lambda$. We know that
$\Area(\tau)$ is maximized when its sides are balanced. Thus
$f(\gamma_\T)$ is maximized on $L$ when $\T$ is nearly balanced.
\end{proof}


\begin{remark}
This is the first of several places where something is shown to be bounded
with reference to a constant $K$.  To avoid proliferating notation, we will 
maintain the symbol $K$ in each successive place that a constant bound
is derived, enlarging it each time as necessary.  No earlier statement will
be hurt by subsequent enlargement, so that in the end one value of $K$ 
depending only on $S$ will suffice for all applications.
\end{remark}

\section{Simple shapes and approximate geodesics}
\subsection{Simple shapes and highest height}
Suppose $u,v\in H(\Z)$ project to integer vectors $\U,\V\in \M$.  We
write $\U\wedge \V$ to denote the determinant of the matrix with those
column vectors, i.e., the area of the parallelogram they define.  Then
when letters $u$ and $v$ are exchanged, the effect on area is given by
the wedge: $z(uv)=z(vu)+\U\wedge \V$.  For instance,
$z(\E_1\E_2)=\frac 12$; $z(\E_2\E_1)=-\frac 12$; and
$\E_1\wedge\E_2=1$.  Note that two group elements commute if and only
if they project to the same direction in the plane.

To keep track of all the possible effects of rearranging letters, we once and 
for all define 
$$N=N(S):=\mathop{\rm lcm}\{\U\wedge\V : u,v\in S\}.$$

\begin{definition} 
For $(a,b)\in\M$, define $n_0(a,b)=|(a,b)|_{\pi(S)}$, so that 
 the fiber $(a,b,*)$ can be reached by 
a spelling path of length $n$ if and only if $n\ge n_0$.
Then for $n\ge n_0$, 
we define the {\em highest height at length $n$
over $(a,b)$} 
 to be the largest $z$ coordinate reachable  
with at most $n$ letters, 
$$w_n=w_n(a,b):=\max\{t : \ell(a,b,t)\le n\}.$$
A spelling (or a group element) will be called {\em highest-height}
if it realizes $(a,b,w_n)$ at length $n$.
Let $W(a,b):=w_{n_0}(a,b)$ be first non-negative $w_n$.  
\end{definition}

Note that $w_n<w_{n+2}$, but that there
may be no spellings at all of a certain parity reaching $(a,b)$, in which 
case $w_n=w_{n+1}$.

\begin{definition}
Given a constant $K$, let $C(K)=\bigcup_{i=0}^K S^i$ 
be the strings in  $S$ whose length is at most $K$
(so that the evaluation map sends $C(K)$ onto the ball of radius $K$ 
in the word metric).
Then a {\em break word} is an element $c\in C(K)$ and 
a {\em break vector} is a tuple of  break words $\C=(c_0,\ldots,c_{\sides })$.

A {\em simple shape} is a tuple $\omega=(\II,\JJ,\B,\C)$, where $\C$
is a break vector, $\II,\JJ$ are indices ($1\le \II,\JJ \le \sides
$), and $\B=(b_1,\ldots,b_{\sides })$ is a vector of integers.  The
{\em simple shape domain} is $\Cone:=\{(s^-,s,s^+)\in \Z^3 : 0\le s^-,s^+ \le s\}$
and the  {\em restricted domain} is $\Cone_0:=\{(s^-,s,0)\} \subset \Cone$.
(Compare to Lemma~\ref{lem::comb-types}.)

Each such shape induces a map from the shape domain to spellings in
the group.  That is, define the evaluation of a simple shape to be
$$\omega(s^-,s,s^+)=c_{\II-1}\sdot \aa_\II^{s^- + b_\ii} \sdot c_\II \sdot \aa_{\II+1}^{s+b_{\ii+1}} \sdot c_{\II+1} \cdots
\aa_{\JJ-1}^{s+b_{\jj-1}}\sdot c_{\JJ-1}\sdot  \aa_{\JJ}^{s^+ + b_\jj}\sdot c_{\JJ},$$
recalling that $\aa_i=a_i^{\sigma_i}$ is defined so that $\isoper=\aa_1\cdots\aa_\sides$. 
Further, we take the convention that if $\II=\JJ$, i.e., if the shape
starts and ends with the same generator, then the domain is restricted
to $\Cone_0$.  
\end{definition}

\begin{example}\label{ex::redundo-gens}
Consider the nonstandard generators for $H(\Z)$ given by
$S=\{a,b,A,B\}^{\pm}$, where $a,b$ are the standard generators and
$A,B$ are big generators $A=a^3,B=b^3$, and a bar denotes the inverse
of an element.  Then the word
$A^{5}aB^9b\bar{a}\bar{A}^{10}\bar{b}\bar{B}^3$ is given by evaluating
the shape with $\C=(e,a,b\bar{a},\bar{b},e)$, $\B=(0,0,1,0)$, $\II=1$,
$\JJ=4$ at $\S=(5,9,3)$.
\end{example}

\begin{remark} There are other shapes with other data that evaluate to
  the same path.  
\end{remark}

\subsection{Bounded difference between word and \CC metrics}\label{sec::bddDiff}

\begin{proposition}[Form for highest-height geodesics]
\label{prop::simpleshapes}
Given a finite generating set $S$, there is 
a number $K=K(S)$ such that any highest-height
spelling path is the evaluation  
of some simple shape with break words from $C(K)$
separating runs of significant letters
given by integer values
$s^-,s^+\le s$ with exponent corrections $0\le b_i\le K$.
\end{proposition}

That is, in a very strong sense, highest-height spellings track
along an arc of a canonical polygon (Busemann's isoperimetrix), 
which has a spelling of the form
$\aa_\II^{s^-}\aa_{\II+1}^s \aa_{\II+2}^s \cdots \aa_{\JJ-1}^s \aa_\JJ^{s^+}$
with $s^-,s^+\le s$, not necessarily integers.  The highest-height spellings
only differ by bounded break words appearing in the corners, and by 
bounded deviation in run lengths.

\begin{proof} 
We suppose that $\gamma$ is a highest-height geodesic over $(a,b)$ and
that its length is $n$.  We claim that the letters of $\gamma$ are in
cyclic order.  If not, we produce $\gamma'$ by putting its letters
into cyclic order.  This changes neither the horizontal endpoint
$(a,b)$ nor the boost $z_b(\gamma)$.  If two of the letters which we
move past each other in this process do not lie in the same direction
in projection, then $z(\gamma') > z(\gamma)$, contradicting our
assumption.  
Thus the letters appearing in $\gamma$ are arranged in cyclic order
in projection.  Also if there are multiple letters $a_i,a_i',a_i''$ projecting 
to the same significant $\A_i$, then clearly $\gamma$ must use the one 
with greatest boost to achieve highest height.

We now claim that there is a bound $K$ on the total exponent of any
non-significant generator.  To see this, suppose that $u$ is a
non-significant generator appearing with large exponent, as a subword
$u^m$.  Supposing $a_i$ and $a_{i+1}$ are the significant generators whose
directions bound the sector that $\U$ lies in, there must be integers $p,q,r$
so that $q\U=p\A_i+r\A_{i+1}$, with $q\ge p+r$.  
We can then replace $u^{kq}$ by  $a_i^{kp} a_{i+1}^{kr}$.  The area
gained by this operation is quadratic in $k$ while any boost lost is
linear in $k$.  Consequently, if $m$ is sufficiently large, this
operation increases height.  So if the total exponent of $u$ in $\gamma$
is $m$, then the reshuffling which brings all powers of $u$ together 
and then performs the subword replacements above will produce a 
path over $(a,b)$ with no greater length and with higher height, 
contradicting the assumption.

It follows now that $\gamma$ consists of corner words of bounded length between
ordered runs of highest-boost significant letters.  That is, we
have
$$\gamma = c_{\II-1} \sdot \aa_\II^{n_\ii} \sdot c_{\II} \sdot \aa_{\II+1} ^{n_{\ii+1}}\dots
c_{\JJ-1} \sdot \aa_{\JJ-1}^{n_\jj}  \sdot c_{\JJ}.$$ 
The statement now follows from an application of the Balancing Lemma
(Lemma~\ref{lem::balancing}).
\end{proof}

\begin{corollary}[Bounded difference]\label{cor::boundedLengthDifference}
For each generating set $S$, there exists a constant $K=K(S)$ with the
following property. If $w_n(a,b) < c \le  w_{n+1}(a,b)$
then $n < |(a,b,c)| \le n+ K$, and if $0\le c\le W=w_{n_0}(a,b)$, then $n_0\le |(a,b,c)| \le n_0+K$.

Consequently, there exists a constant $K=K(S)$ such that 
$$\dCC(\X,\zero) - K  \le |\X|_S\le \dCC(\X,\zero) + K.$$
\end{corollary}

Put differently, the embedding of $H(\Z)$ with generating set $S$ into
$H(\R)$ with the corresponding \CC metric is a $(1, K)$
quasi-isometry. 

\begin{proof}
By definition $(a,b,w_n)$ is the highest-height element of the fiber
over $(a,b)$ which can be reached by a spelling of length less than or
equal to $n$, so $n < |(a,b,c)|$.

Let $\omega(s^-,s,s^+)$ and $\omega'(t^-,t,t^+)$ be shapes evaluating
to geodesic spellings for $g=(a,b,w_n)$ and $g'=(a,b,w_{n+1})$.  Let
$\tau$ and $\tau'$ be $\isoper$-arcs which fellow-travel these in
projection (whose existence is guaranteed by the previous result). 
If $\tau$ and $\tau'$ are of almost the same combinatorial
type, then the polygonal paths $\beta=\pi(\omega(s^-,s,s^+))$ and
$\beta'=\pi(\omega'(t^-,t,t^+))$ fellow-travel.  Consider the sequence of
paths $\beta' = \beta_0,\beta_1,\dots,\beta_{n+1}= \beta$ formed as follows.  For
$i=1,\dots n$, let $\beta_i$ be the path starting along $\beta$ until $\beta(i)$,
taking a geodesic from $\beta(i)$ to $\beta'(i)$, and continuing along
$\beta'$. Since $\beta$ and $\beta'$ $K_0$--fellow-travel for some $K_0$, 
the connecting geodesics have
bounded length, so each $\beta_i$ has length at most $n+1+K_0$.  Take
$\gamma_i$ to be the lift of $\beta_i$.  These $\gamma_i$ end at group
elements $(a,b,c_i)$ with $|c_{i+1}-c_i| \le 2K_0+2$.  
Thus any value $(a,b,c)$ in the range in question can be reached by tacking
a bounded-length path on to the end of an appropriate $\gamma_i$.
It follows that
there is $K$ such that for each $c$ with $w_n <c \le w_{n+1}$,
$|(a,b,c)| \le n +K$ as required.

If $\tau$ and $\tau'$ are not of almost the same combinatorial type,
then by Lemma~\ref{lem::comb-types} $(a,b)$ is close to the origin and
$\tau$ and $\tau'$ almost complete the entire boundary of an
isoperimetrix.  It follows that we can replace $\omega'(t^-,t,t^+)$ by
a spelling path which fellow-travels $\omega(s^-,s,s^+)$ in projection and is
only boundedly longer than $\omega'(t^-,t,t^+)$.  
(To be concrete, the blocks of significant letters and the corners can be 
preserved but reordered to correspond to the combinatorics of $\tau'$.)

Note that for an $\isoper$--arc of spelling length $n$, its length 
 the $L$--norm is $n$, so its lift has \CC length $n$ as well and it is geodesic.
 Therefore $\dCC\left( (a,b,w_n),\zero\right)$ is boundedly close to $n$ and we are done
 with the case $w_n<c\le w_{n+1}$.
 
For heights below $W$, we begin with a highest-height spelling realizing $(a,b,w_{n_0})$.
By permuting the letters, we can lower the height in bounded increments down to some minimum.
Suppose it can be lowered to a non-positive height.
Then since the intermediate heights can be reached by appending a bounded-length correction word,
we have $n_0\le |(a,b,c)| \le n_0+K$.  On the other hand, by Lemma~\ref{cc-types}, the 
\CC distance from $\zero$ is constant in the $(a,b)$ fiber up to the first height reached by a regular 
geodesic, which is boundedly close to $(a,b,W)$.
We enlarge the constant $K$ from the first statement in the Lemma to be sufficient for the 
second statement.

On the other hand, it may be that every permutation of the letters in the spelling
has positive height, for instance if the spelling is simply a single repeated letter with positive boost.
In this case, suppose that $a_{\II+1}$ is the first significant letter in the spelling and 
choose some significant letter $u$ such that $\U\wedge\A_{\II+1}<0$.  It follows that there is some 
power of $u$ such that the conjugate $u^k a_{\II+1} u^{-k}$ has height below zero.  From this 
modified word, complete the proof as before with successive permutations.

Finally, observe that the map $g\mapsto g^{-1}$ is a length-preserving
bijection which carries $(a,b,c)$ to $(-a,-b,-c)$, so the $c<0$ case is similar.
\end{proof}

This gives a new proof of Krat's result.
And in particular, since Krat's theorem (bounded difference) 
has a stronger conclusion than Pansu's theorem
(ratio goes to 1), our argument also gives a direct geometric proof of Pansu's 
theorem for the special case of arbitrary word metrics on $H(\Z)$.

\subsection{Simplification}

We will see below that every regular element has a geodesic which is
close to a simple shape.  To this end, we show that we can modify paths to become simple
shapes while staying in the same fiber and increasing height in a
controlled manner.

\begin{lemma}[Simplifying paths]\label{lem::simplification}
There is a constant $K=K(S)$ 
so that for each spelling path $\gamma$ there exists a refined path $\gamma_1$ 
with the following properties.
\begin{itemize}
\item If $(a,b,c)$ is the evaluation of $\gamma$, then $\gamma_1$ evaluates
to $(a,b,c+kN)$ for some $k\ge 0$; 
\item the length of $\gamma_1$ is less than or equal to the length of $\gamma$;
\item $\gamma_1=\omega(s^-,s,s^+)$ for some simple shape $\omega$ with 
corners from $C(K)$.
\end{itemize}
\end{lemma}

\begin{proof}
Suppose that $u,v$ are any two letters appearing in $\gamma$ such that
$\U \wedge \V > 0$ (so that $u$ comes before $v$ in the cyclic
ordering of their projections, and replacing $vu$ with $uv$ increases
area).  Let $\Lambda_{u,v}=\Lambda_{u,v}(\gamma)$ be the sum of all of
the exponents $k$ appearing in distinct subwords $vu^kw$ of $\gamma$
with $w\in S$, $w\neq u$.  Then we can make generator swaps of $u$ and $v$
letters to change the height by any multiple of $\U\wedge \V$ less
than or equal to $ \Lambda_{u,v}\sdot \U \wedge \V$.  By rounding
$\Lambda_{u,v} \sdot \U \wedge \V$ down to the nearest multiple of
$N$, we can perform generator swaps to obtain $\gamma_1$, so that
$\Lambda_{u,v}(\gamma_1)\le N$ for all pairs $u,v$.

Notice that we may have to perform this procedure many times.  A
single application of this procedure reduces $\Lambda_{u,v}$ to be
less than $N$, but may increase $\Lambda_{u,v'}$.  We can perform this
procedure whenever there is some pair $u,v$ so that 
$\Lambda_{u,v}\sdot \U \wedge \V > N$.  We claim that repeated
applications of procedure must eventually terminate with a spelling
where there is no such pair.  To see this, consider the total number
of pairs of letters in the spelling which are out of order.  This
total number decreases at every application of the procedure, and
hence we must terminate with $\Lambda_{u,v}\sdot \U \wedge \V < N$ for
every pair $u$ and $v$.

Next, we will cash in any big blocks of non-significant letters for significant letters.
Recall that significant letters project to corner points of the
polygon $L$, while edge letters project to other boundary points and
interior letters project to the interior.  That is, for an edge letter
$u$ and an interior letter $v$ with projections in the sector between
$\A_i$ and $\A_{i+1}$, we have $q\U=p\A_i+r\A_{i+1}$ and
$q'\V=p'\A_i+r'\A_{i+1}$ such that $q=p+r$ while $q'>p'+r'$.

Consider the subword replacements $u^{kNq} \to 
a_{i}^{kNp} a_{i+1}^{kNr}$, 
or $v^{kNq'} \to a_{i}^{kNp'} a_{i+1}^{kNr'}$.
As above, the new paths  reach the same endpoint in $\M$ while
either preserving or reducing the total spelling length of the path,
gaining area by an amount proportional to $k^2$, and reducing boost
by an amount proportional to $k$.  
We perform these replacements in every instance where
 $k$  is large enough to produce is a net height increase; 
and note that the height change is a multiple of $N$.

Repeat the reordering and the replacement steps one after the other until 
neither can be performed any further.
Then $z(\gamma_1)\ge z(\gamma)$, and they differ by a linear
combination of the wedges; namely, $\Delta z =
z(\gamma_1)-z(\gamma)=\sum_{u,v} k_{uv} (\U\wedge \V)$, for integers
$k_{uv}\ge 0$ with $k_{uv}\equiv 0 \pmod N$.
At this stage, the path $\gamma_1$ has well-defined sides with mostly $a_i$ letters
and only boundedly many exceptions.

Next, we set things up to push the remaining ``out of place" letters to the corners so that the 
$a_i$ side is mostly a single long block of the $a_i$ letter.
So far we  have a spelling $\gamma_1$ that contains boundedly many non-significant
letters and boundedly many significant letters on the wrong side. 
Consider a side which consists of significant
generator $a_i$ with a bounded number of letters which are not $a_i$.
For each letter $u$ on the $a_i$ side, the sign of $\U\wedge \A_i$ tells
us whether replacing $a_i u$ with $ua_i$ is height-increasing or height-decreasing
(note that the case $\U\wedge\A_i=0$ is the case that $u$ and $a_i$ commute).
We swap each $u$ past $a_i^N$ in the height-increasing direction (or an arbitrary direction
if they commute)
until we create a block $w=u'a_i^m u$ with $m<N$.  This $w$ itself can be commuted
with $a_i^N$ to the left or right, not decreasing height.  
Since there are boundedly many out
of place letters on each side, this process ends with all these
letters within a bounded distance of a corner, so we merge them with the corner words.  
At each move we have
increased height by a multiple of $N$.

Finally, we balance the side lengths of
$\gamma_1$.  To do this we apply the balancing lemma (Lemma~\ref{lem::balancing})
to the lattice of integer tuples which differ from 
the original $(t_\II,t_{\II+1},\dots,t_\JJ)$ by multiples of $N$
in each coordinate.  This ensures that area and boost, and therefore height, changes
by a multiple of $N$.

This final step has produced a modified path, again called $\gamma_1$, which 
still has the same $(a,b)$ endpoint as $\gamma$ and may have 
higher height by a multiple of $N$.  Now there 
are  bounded-size exceptional corner
words between the sides, and the exponents of significant 
blocks differ only by a bounded amount, so this is the evaluation of a simple shape.
\end{proof}

\section{General shapes and unsimplification}

\subsection{General shapes}

Beyond simple shapes, we will need a construction of shapes with 
break words not only at the corners: 
runs of significant generators can be separated by finitely many other break words.

\begin{definition}
Given a generating set $S$ for which the isoperimetrix has $\sides$ sides, 
a {\em general shape} with parameter $K\ge 1$
is a tuple  $\omega=(\II,\JJ,\B,\chi)$, where 
\begin{itemize}
\item $1\le \II,\JJ\le  \sides $ are a starting and ending side;
\item $\B=(b_1,\ldots,b_{\sides })$ is a vector of integers $0\le b_i \le K$; 
\item $\chi$ is a $(K-1)\times \sides $ matrix whose entries are 
break words from $C(K)$.
\end{itemize}

Let $\Shape_K$ be the set of all such shapes, clearly a finite set for
each value $K$.  We will evaluate each shape at a matrix $X\in
M_{K\times\sides}$.  Let $\Lambda: \Shape_K \times M_{K\times \sides}
\to \Z^{\sides}$ be given by
$\Lambda(\omega,X)=(\lambda_1,\ldots,\lambda_{\sides})$, where
$\lambda_i:=\left(\sum_{j=1}^K x_{ji}\right)-b_j$.  Then the {\em
  shape domain} $\Dom_K(\omega)$, for $\omega\in\Shape_K$, is the set
of $K\times \sides $ matrices $X$ of non-negative integers satisfying
a condition on the image of $\Lambda$, namely:
\begin{itemize}
\item  $\lambda_i=\lambda_j$ for all 
$\II<i,j<\JJ$;
\item $\lambda_\II,\lambda_\JJ\le \lambda_i$;
\item $\lambda_t=0$ for the $t$ that are not between $\II$ and $\JJ$;
\item if $\II=\JJ$, then $\lambda_JJ=0$.
\end{itemize}
With slight abuse of notation, we will then write 
$\Lambda:\Shape_K\times \Dom_K \to \Cone$ given by
$\Lambda(\omega,X)=(s^-,s,s^+)$ where $s^-=\lambda_\II$, $s=\lambda_{\II+1}=\dots=\lambda_{\JJ-1}$, 
and $s^+=\lambda_\JJ$.  (Note that the last condition in the definition ensures that 
the map lands in $\Cone_0$ in the $\II=\JJ$ case.)
\end{definition}

It is immediate from this definition that $\Dom_K(\omega)$ is given by
pulling back a rational family under an affine map.

The matrix $X$ is to be thought of as a matrix of {\em run lengths}.
The evaluation of a shape, $\omega(X)$, is the concatenation of the
break words with the runs of significant generator blocks of length
prescribed by $X$.  The $\B$ vector records the failure of the column
sums to be equal, i.e., the failure of the shadow to be balanced in
terms of its side lengths.  (Since its entries are bounded, the
column sums are nearly equal, which means that the spelling will track
close to an isoperimetrix.)  Simple
shapes are a subset of general shapes for which the break words only
appear at the corners.

\begin{remark}
Note that the triple $(a,b,\ell)$ associated to a spelling $\omega(X)$
factors through $\Lambda$.  That is, as $X$ ranges over
$\Dom_K(\omega)$,  the three integers $\Lambda(\omega,X)=(s^-,s,s^+)$ 
determine the horizontal position and the
word length of the evaluation word.  Thus we can regard this as a map
$\omega:\Cone\to\Z^3$ that is affine and injective.  
\end{remark}

\begin{remark}
If significant generators include several options with same projection
and different boost, then we also need $Y$, a matrix specifying for
each side how many of each different boost level get used, and in this
case the evaluation will be $\omega(X,Y)$.  This makes no meaningful
difference anywhere in the argument below.
\end{remark}

\subsection{Unsimplification}

We describe a {\em 2-sided surgery} and a
{\em 3-sided surgery} 
for paths
and then explain how to use them algorithmically
to begin with a path described by a simple shape and produce a path 
ending lower in the same fiber and 
still described by a general shape.
In both of these moves, we will suppose that
$a_1,a_2,a_3$ are successive significant generators and that $p,q,r$
are the values with $\gcd=1$ so that where $q\A_2=p\A_1+r\A_3$.  (In
the special case that $\A_1=-\A_3$ (the parallel case), we have such a
surgery with $p=r=1$, $q=0$.)  Here we describe the surgeries on side
$a_2$.

{\bf 2--sided surgery.}  Here, a subword of the form
$a_1^{s_1}c_1a_2^{s_2}$ is replaced by $a_1^{s_1-3Np}c_1wa_2$, where
$w$ is a permutation of the letters in $a_1^{3Np}a_2^{s_2-1}$.

{\bf 3--sided surgery.}  Here, a subword of the form
$a_1^{s_1}c_1a_2^{s_2}c_2a_3^{s_3}$ is replaced by
$a_1^{s_1-2Np}c_1a_2^{s_2+2Nq}c_2a_3^{s_3-2Nr}$

\begin{figure}[ht]
\begin{tikzpicture}[scale=.4]
\begin{scope}[xshift=14cm]
\filldraw [gray!20] (2,0)--(6,0)--(8,1)--(12,5)--(13,8)--(13,12)--(12,9)--(8,5)--(4,1)--cycle;
\node at (4,0) [below] {$a_1^{s_1}$};
\node at (7,0.5) [below right] {$c_1$};
\node at (10,3) [below right] {$a_2^{s_2}$};
\node at (12.5,6.5) [right] {$c_2$};
\node at (13,10) [right] {$a_3^{s_3}$};
\draw [gray,dashed] (4,1)--(8,1) (8,5)--(12,5)--(12,9);
\node at (10,5) {$a_1^{2Np}$};
\draw (0,0)--(6,0)--(8,1)--(12,5)--(13,8)--(13,15);
\draw [blue,thick] (2,0)--(4,1)--(8,5)--(12,9)--(13,12);
\end{scope}
\begin{scope}[yshift=2cm]
\filldraw [gray!20]  (2,0)--(4,1)--(5,1)--(6,2)--(7,2)--(11,6)--(14,6)--(9,1)--(7,0)--cycle;
\node at (3.5,0) [below] {$a_1^{s_1}$};
\node at (8,.5) [below right] {$c_1$};
\node at (11.5,3.5) [below right] {$a_2^{s_2}$};
\draw (0,0)--(7,0)--(9,1)--(16,8);
\draw [dashed] (4,1)--(10,7)--(15,7);
\node at (12.5,7) [above] {$a_1^{3Np}$};
\draw [gray,dashed] (5,1)--(9,1);
\draw [blue,thick] (2,0)--(4,1)--(5,1)--(6,2)--(7,2)--(11,6)--(14,6);
\end{scope}
\end{tikzpicture}
\caption{Examples of 2-sided and 3-sided surgery with corners.  If the length $s_2$ of the second side is long enough,
then 2-sided surgery can make a larger change to area because it is thicker:  the width of the surgery is proportional to $3N$ rather than $2N$.}
\end{figure}
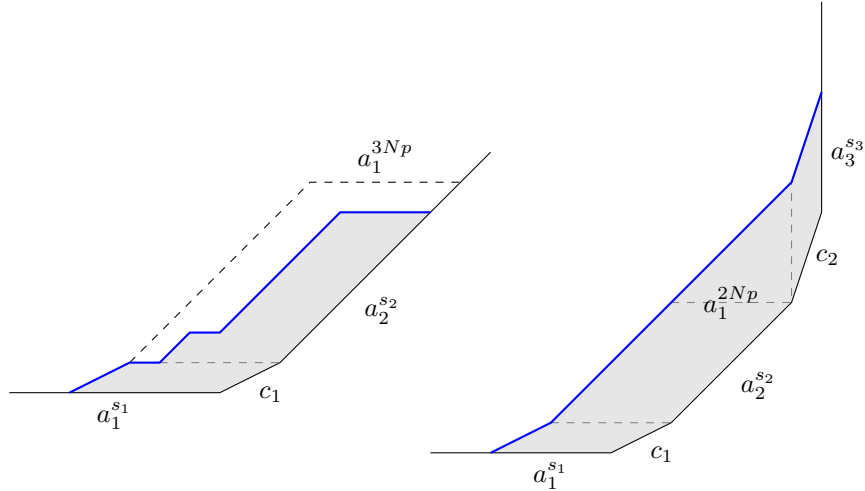

\begin{lemma}[Unsimplification for shapes]\label{lem::unsimp-shape}
Given a starting word $\gamma$, let $\gamma_1$ be the simplification described above and suppose the 
height difference $\Delta z = z_1-z_0$ is sufficiently large.  
Then for any full side of $\gamma_1$,  
a sequence of (possibly zero) $3$-sided surgeries on that side 
followed by at most one 
$2$-sided surgery on that side produces a word $\gamma_2$ which evaluates to the same group element as $\gamma$.
\end{lemma}

Note that if there are fewer than three sides (so that there is no well-defined ``full side"), then we 
can appeal to the unstable (pattern) case presented in the next section.

\begin{proof}
The change in area for each application of three-sided surgery equals 
$$(3SS):=2Np(\A_1\wedge \C_1)+2Nps_2(\A_1\wedge\A_2)+2N^2pr(\A_1\wedge\A_3)+2Nr(\C_2\wedge \A_3).$$
We note that since the wedges are all integers, this is divisible by $N$ and therefore also by $\A_1\wedge\A_2$.

On the other hand, the area difference from performing two-sided surgery depends on the permutation parameter; the area change equals
$$(2SS)_k:=3Np(\A_1\wedge \C_1) + k(\A_1\wedge\A_2),$$
where $k$ is an arbitrary integer, $0\le k\le 3Np(s_2-1)$.  

The lemma's assumption that the height difference is large enough can be taken to precisely mean that 
$\Delta z = z_3-z_0>(2SS)_0$.

Perform (3SS) repeatedly, updating $z_3$ each time, until
$$\Delta z < (2SS)_0+(3SS).$$ 
Then we must show that there exists $k$
such that $\Delta z =(2SS)_k$.  We know that $\Delta z$ is a multiple
of $N$ and therefore of $\A_1\wedge\A_2$.  On the other hand, $\Delta
z $ is greater than $(2SS)_0$, and $(2SS)_k$ achieves all multiples of
$\A_1\wedge\A_2$ past that threshold and up to its maximum.  Thus it
is enough to show that $\Delta z <(2SS)_{\rm max}$.  Since we saw
above that $\Delta z <(2SS)_0+(3SS)$, this amounts to showing that
$(2SS)_{\rm max}-(2SS)_0>(3SS)$.  Since all the wedges of vectors and
the values $p,q,r$ are fixed by the choice of side, it suffices to
take $s_2$ sufficiently large: since the left-hand side has a term
$3Nps_2(\A_1\wedge\A_2)$ and the right-hand side has a term
$2Nps_2(\A_1\wedge\A_2)$, eventually the difference between these
overwhelms all the other fixed terms.
\end{proof}

\section{Patterns}

Recall that \CC geodesics are classified into two kinds (see
Sec.~\ref{sec::CC}), regular and unstable. In a particular fiber
$(a,b,*)$, only unstable geodesics reach positive heights below a
certain threshold height and only regular geodesics reach above that
level.  We will consider the corresponding situation for word
geodesics.

We defined $W= W(a,b)$ to be the highest height reached by a spelling
path of length $n_0 = |(a,b)|_{\pi(S)}$.  In each fiber $\{(a,b,*)\}$,
the general shapes defined in the previous section will reach the
elements  $\{(a,b,c) : c > W\}$, which may be called
{\em regular} elements.  In this section we turn to the growth of the
{\em unstable} elements. Here we will consider the
unstable elements $\{(a,b,c) : 0 \le c \le W\}$ at non-negative heights.  
(Later, we will appeal
to the map $g\mapsto g^{-1}$ which carries $(a,b,c)$ to $(-a,-b,-c)$
to deal with the negative heights.)

\begin{definition} A {\em pattern} is a tuple $\W=(\II, c_1,c_2,c_3)$,
where each $c_i\in C(K)$ is a break word (a string of length at most $K$),
and $1\le \II\le \sides$ picks out a sector between successive significant
directions $\A_\II$, $\A_{\II+1}$.  The (finite) set of all such patterns will be denoted
$\Patt_K$.
 Each pattern $\W$ gives a map $\N^2 \to S^*$ via 
$\W(n_1,n_2) = c_1 a_\II^{n_1} c_2 a_{\II+1}^{n_2}c_3$.
\end{definition}

\begin{lemma}[Simplifying to a pattern]
Let $(a,b)$ lie in the sector between $\A_\II$ and $\A_{\II+1}$, and let $N$
be the lcm of all possible area swaps, as usual.  
Then there is $K=K(S)$ with the following property.
If $\gamma$ is a geodesic for an unstable element $(a,b,c)$, then there is a pattern
$\W\in\Patt_K$ and $n_1,n_2\in \N$ such that 
the spelling path $\tau = \W(n_1,n_2)= c_1 a_\II^{n_1} c_2
a_{\II+1}^{n_2}c_3$ has the following properties:
\begin{itemize}
\item the paths $\tau$ and $\gamma$ have the same length; 
\item the path  $\tau$  evaluates to an  element $(a,b,c+kN)$
with $k\ge 0$ and $c+kN\le W+K(n_1+n_2)$; and
\item the letters $a_\II$ and $a_{\II+1}$ are the highest-boost generators projecting to
$\A_\II$ and $\A_{\II+1}$, respectively.
\end{itemize}
\end{lemma}

\begin{proof}
First note that 
if that $\gamma$ is a geodesic for an unstable element $(a,b,c)$
where $(a,b)$ lies in the sector between $\A_i$ and $\A_{i+1}$, then
all but boundedly many letters in $\gamma$ project to convex combinations
of $\A_i$ and  $\A_{i+1}$ .
This is because, 
by Bounded Difference (Corollary~\ref{cor::boundedLengthDifference}), there is
$K$ such that if $(a,b,c)$ is unstable, then $n_0 \le |(a,b,c)| \le
n_0+K$, so that the projection $\pi(\gamma)$ must reach $(a,b)$ in at
most $n_0+K$ letters.  This means $\pi(\gamma)$ can only use boundedly many letters
that are not on the edge between those points (i.e., convex combinations of
$\A_i$ and $\A_{i+1}$)---to see this, just consider 
orthogonal projection to the normal of that edge, so every time any other letter
is used, the projection falls behind by a definite amount.

We now carry out the simplification procedure used above
(Lemma~\ref{lem::simplification}), making a few extra observations as
we go.  We note that the length of the path in this case will be
maintained and not shortened, because there are only boundedly many
interior letters and so we need not cash them in for significant
letters.

If $a_i$ is the highest-boost lift and $a_i'$ is another letter
projecting to $\A_i$, then we can replace any 
$(a_i')^N$ by $a_i^N$.  The remaining (boundedly many) $a_i'$, which commute
with $a_i$, can be pushed to the corner position.

Finally, $\pi(\tau)$ fellow-travels the $L$--norm geodesic
$\A_\II^{n_1}\A_{\II+1}^{n_2}$, and therefore fellow-travels any geodesic achieving
$(a,b,W)$, so
$|z(\tau)-W|$ is bounded by a constant multiple of  $n_1+n_2$.
We enlarge $K$ if necessary to complete the lemma.
\end{proof}

On the other hand, by controlled rearrangement of letters, patterns
can produce a range of group elements in the same fiber.

\begin{definition} For a pattern $\W=c_1 a_1^{n_1} c_2 a_2^{n_2} c_3$
evaluating to $(a,b,c)$, define a process of rearrangements as
follows.  Consider letters $b_1,\dots,b_k$ appearing in the word
$c_2$.  For $j=1,\dots,k$, let $d_j = \frac{N}{\A_{\ii+1} \wedge
  \B_j}$, so that commuting $a_{\II+1}^{d_j}$ through $b_j$ decreases
height by $N$.  We greedily perform commutations to move $a_{\II+1}$
letters past $c_2$, then continue if possible by commuting groups of
$a_{\II+1}$ letters through $a_\II$ letters.  Consider the set of
$(a,b,c')$ achievable by this process for which $0\le c'\le W$, and
let the {\em height interval} of the pattern, denoted $\III_\W(a,b)$,
be the $z$ coordinates in this set. Note that by
  construction $\III_\W(a,b)$ is the intersection of an interval with
  a residue class.
\end{definition}

For example, for the generators $\{a,b,A,B\}^\pm$  described above
in Example~\ref{ex::redundo-gens},
if $\W=aA^*aB^*b$, then $\III_\W(52,131)=\{6,106,206,\dots,3406\}$.  
Here $W(52,131)=3406$ and $N=100$.

\begin{lemma}[Unsimplification for patterns]\label{lem::unsimp-patt}
Let $\W(n_1,n_2)$ evaluate to $(a,b,C_\W)$, and define $C'_\W=\max{\III_\W}$ 
and $C''_\W=\min{\III_\W}$, so that $C_\W,C_\W',C_\W''$ are functions
of $n_1,n_2$ (or equivalently of $a,b$) representing the possible heights of rearrangements of
patterns.
Then there is a partition of $\N^2$ given by finitely many linear equations,
inequalities, and congruences such that $C_\W,C_\W',C_\W''$ are
given by quadratic polynomials in $n_1,n_2$ on each set in the partition.
Therefore there is a corresponding partition of $\M$ so that these heights
are quadratic on each piece on which $\III_\W\neq \emptyset$.
\end{lemma}

\begin{proof}
Fixing $\W$, the height $C_\W$ can be seen as a function of
$(n_1,n_2)$ whose degree-two term equals $\frac 12 n_1 n_2 (\A_\II
\wedge \A_{\II+1})$, because $\W(n_1,n_2)$ fellow travels the
two-sided figure $\A_\II^{n_1}\A_{\II+1}^{n_2}$.  Fellow traveling
ensures that the enclosed areas differ by at most an amount
proportional to the length of the shape plus the boost provided by
corner words, which are terms of degree one and zero.  

$W$ is the highest height of a minimal-length spelling path reaching 
the shadow of $\W(n_1,n_2)$.  
The simplification argument above shows that 
the spelling path realizing height $W$ must also be boundedly close in projection
to $a_\II^{n_1}a_{\II+1}^{n_2}$, so the difference $W-C_\W$ is a 
linear function as well.  If it is positive, then $C_\W'=C_\W$.  
if it is negative, then $C_\W'$ are given by quadratic polynomials on each 
residue class of $C_\W\pmod N$.

The lowering process can take the pattern all the way down below height zero
as long as $n_2$ is sufficiently large compared to $N$.  If it is not, then 
the quadratic expression for $C_\W''$ in terms of $n_1,n_2$
is given by linear functions of $n_1$ for each small value of $n_2$.  

Finally, the  
$(a,b)$ are linearly related to $(n_1,n_2)$ via $(a,b)=n_1\A_\II + n_2
\A_{\II+1} +\overline{\C}$, where $\overline{\C}$ is the sum of the
corner words, so a change of basis finishes the proof.
\end{proof}

\section{Word geodesics tracking close to \CC geodesics}

\begin{theorem}[Realization by shapes and patterns]\label{thm::shapes}  For every
  generating set $S$, the following two equivalent conditions hold:
\begin{itemize} 
\item there is a $K=K(S)$ such that every group element has a geodesic spelling for which the shadow is $K$--close 
to the shadow of a \CC geodesic;
\item there is a $K=K(S)$ such that every group element has a geodesic  spelling which is 
either the rearrangement of some pattern from $\Patt_K$ 
or the evaluation of some general shape
  from $\Shape_K$.
\end{itemize}
\end{theorem}

\begin{proof}
Suppose $\gamma$ is a geodesic spelling in $(H(\Z),S)$ evaluating to 
 $(a,b,c)\in H(\Z)$. 
 Recall  that $N$ was defined as the least common multiple of the areas spanned by pairs
of letters in the generating alphabet.
Then any single neighboring
generator-swap suffices, if performed enough times, to produce area
changes of any multiple of $N$.
The steps will be organized to ensure that, though the height may
change, it stays in the same residue class modulo $N$.  Throughout, we
will be assuming $n=\ell(\gamma)\gg N$.

First we simplify $\gamma$ to $\gamma_1$ (Lemma~\ref{lem::simplification})
by shuffling letters, cashing in insignificant generators, and balancing lengths.  
  We know that $\gamma_1$ is 
$K$--almost balanced with respect to the induced norm on $\M$.
This means that it has the form $c_{\II-1}a_\II^{n_\ii}c_\II
\cdots a_\JJ^{n_\jj}c_\JJ$ and that for $\II <i <j < \JJ$, the values
$\frac{n_i}{\sigma_i}$ and $\frac{n_j}{\sigma_j}$ differ by at most a
bounded amount and that the values $\frac{n_\ii}{\sigma_\ii}$ and
$\frac{n_\jj}{\sigma_\jj}$ can exceed these by at most a bounded
amount, though of course these values are not necessarily integral.
We can rewrite such a  spelling $\gamma_1$ as
$$\gamma_1 = c'_{\II-1}\aa_\II^{s^-}c'_\II\aa_{\II+1}^s \cdots
\aa_{\JJ-1}^s c'_{\JJ-1}\aa_\JJ^{s^+} c'_{\JJ}.$$
In particular the projection $\pi(\gamma_1)$ fellow-travels a \CC
geodesic.

Depending on the number of sides, we 
next apply unsimplification for shapes or patterns 
(Lemma~\ref{lem::unsimp-shape} or \ref{lem::unsimp-patt}) to obtain $\gamma_2$.
In the pattern case, note that $(a,b,c)$ is geodesically spelled by some 
rearrangement of the pattern $\W$, because the pattern was obtained in the first place by shuffling 
the original spelling.

To complete the proof of the Theorem for shapes, we observe that we are in one of
three cases: either the height difference $z_2-z_0<(2SS)_0$ so that we can not apply
unsimplifcation; the unsimplification process had at least one three-sided
surgery; or unsimplification had only two-sided surgery.  If any
three-sided surgery was performed, then our new spelling $\gamma_2$ 
 evaluates to the same word as $\gamma$ but is
shorter, contradicting geodesity of $\gamma$.  If only two-sided
surgery was needed, then a $\gamma_2$ of equal length to $\gamma$ has
been produced, but with lower eccentricity.  Finally, if $z_2-z_0$ is
smaller than some fixed bound, then the steps in the proof only made minor changes
to $\gamma$, and retracing the argument
this implies that $\gamma$ was  boundedly close to
isoperimetric at the beginning of the process.
\end{proof}

\begin{example} We will run an example to illustrate an 
eccentric word geodesic being improved
by the shape algorithm above.
Consider the standard generators, fix a value $D$ and take $M\gg D$.   
Let $\gamma$ be the closed rectangular path 
$$\E_1^{M-D}\E_2^{M+D}\E_1^{-M+D}\E_2^{-M-D}.$$
This has length $4M$ and encloses area $M^2-D^2$, so it evaluates to the group element
$(0,0,M^2-D^2)$.  The \CC geodesic reaching the same element would have length 
$4\sqrt{M^2-D^2}$, which is strictly greater than $4M-1$ if $M$ is large enough 
compared to $D$, and this means that $\gamma$ is a geodesic.
It is already cyclically ordered and has no out-of-place letters, so 
$\gamma_1=\gamma$.  Balancing the sides produces 
$\gamma_2=\E_1^{M}\E_2^{M}\E_1^{-M}\E_2^{-M}$, which has area $M^2$.
Now we perform a 2-sided surgery, replacing $\E_1^M\E_2^M$ with 
$\E_1^{M-1}\E_2^{D^2}\E_1\E_2^{M^2-D^2}$.  This reduces the area by $D^2$
while preserving length,
so creates a geodesic to $(0,0,M^2-D^2)$
that $1$--fellow-travels the \CC geodesic.
\end{example}

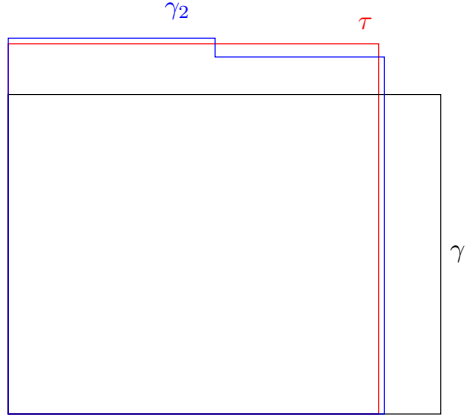
\begin{figure}[ht]
\begin{tikzpicture}[scale=.25]

\draw [red] (0,0)--(19.7,0)--  (19.7,19.7)-- (0,19.7)--cycle;
\node [red] at (19,20) [above] {$\tau$};
\draw (0,0)--(23,0)--  node [right] {$\gamma$} (23,17)-- (0,17)--cycle;
\draw [blue] (0,0)--(20,0)--(20,19)--(11,19)--(11,20)--(0,20)--cycle;
\node [blue] at (9,20.5) [above] {$\gamma_2$};
\end{tikzpicture}
\caption{Here, a word geodesic $\gamma$ with large eccentricity is shown compared to the corresponding \CC
geodesic $\tau$,  which can't be realized with integers.  The algorithm {\em balances} $\gamma$
and then {\em chips away} area to produce
a geodesic $\gamma_2$ which evaluates to the same group element as the original
$\gamma$ but tracks close to $\tau$.}
\end{figure}

\section{Picking out geodesics}\label{sec::competition}

\subsection{Linear comparison for shapes}

We have seen that when $\omega$ is a shape (simple or general), the
map $\Dom_K \to \Z^3$ induced by $\omega$ taking $X\mapsto (s^-,s,s^+)\mapsto
(a,b,\ell)$ is injective and affine.  Therefore, for a given shape
$\omega$, the inverse map $(a,b,\ell)\mapsto \S=(s^-,s,s^+)$ is an
affine function on $\omega(\Dom_K)\subset \Z^3$.

First, we define the {\em domain of competition} for a pair of shapes 
to be the inputs for which they reach the same horizontal position at nearby lengths:
$$\DomComp_K(\omega,\omega') = \left\{(X,X')\in \Dom_K(\omega)\times \Dom_K(\omega') :
|\ell-\ell'| \le K\right\}$$

Define a {\em competition function} $f_{\omega\omega'}:
\DomComp(\omega,\omega')\to \Z$ to be the difference in heights, 
$z(\omega(X))-z(\omega'(X'))$.
We show that if two shapes ever compete, then the domain of competition 
decomposes into rational
families where that height difference is given by a linear function.

\begin{definition} 
Given a general shape $\omega$ of type $(\II,\JJ)$ and data $X$ with
lengths $\S=(s^-,s,s^+)$, we define the {\em trace}
$\tau=\tau(\omega(X))$ to be the corresponding $\isoper$--arc
$$\tau=\AA_\II^{s^-}\AA_{\II+1}^{s}\dots \AA_{\JJ-1}^s \AA_\JJ^{s^+}.$$
\end{definition}

That is, $\tau$ is equal to $\omega(X)$ with the break words deleted
and the exponent differentials erased.  Observe that by construction,
\begin{itemize}
\item the dependence of $\tau$ on $X$ factors through
  $(s^-,s,s^+)$;
\item $\tau$ begins at $\zero\in \M$, and synchronously fellow-travels
  $\omega(X)$ with a fellow-traveller constant which depends only on
  $\omega$ and is independent of $X$; and
\item for a given $\omega$ the difference between the endpoint of $\omega(X)$
  and the endpoint of $\tau$ is independent of $X$, i.e.,
  is constant on $\Dom_K(\omega)$.  This is because this difference depends only on the 
  $\B$ and $\C$ data from
$\omega$.
\end{itemize}

\begin{lemma}[Linear comparison for shapes]\label{linearComparison}
If $\DomComp_K(\omega,\omega')$ is nonempty, then there is a finite
partition such that 
each piece $U_\delta \subset \DomComp_K(\omega,\omega')$ is
defined by linear equations,  linear inequalities, and congruences, and
the comparison function 
$f_{\omega\omega'}\vert_{U_\delta}$ is linear.
\end{lemma}

\begin{proof} Take $K$ to be the bounded-difference constant 
from Corollary~\ref{cor::boundedLengthDifference}.
We will partition the domain of competition into pieces for each
$-K\le \delta \le K$ consisting of the subset of positions $(a,b)$
reached by $\omega$ at some length $\ell$ and by $\omega'$ at length
$\ell+\delta$.  Call this subset $U_\delta$.  Let $(s^-,s,s^+)$ and
$(t^-,t,t^+)$  denote the length data extracted from $X$ and $X'$
respectively.

Fixing this $\delta$, 
we first consider the case where $\omega$ and $\omega'$ have the same
combinatorial type, that is, $\II=\II'$, $\JJ=\JJ'$.  Further, if
$\II=\JJ$, recall that we have restricted the domain so that $s^+ =
t^+ = 0$.  We claim that the trace  $\tau'$ fellow
travels $\tau$ in projection and that the distance between
corresponding sides is independent of $a$, $b$ and $\ell$.  This is
because the affine maps
\begin{align*}
 (s^-,s,s^+) &\mapsto (a,b,\ell) \\ 
  (t^-,t,t^+)  &\mapsto (a,b,\ell+\delta)
\end{align*}
have the same linear part, hence so do their inverses.  Thus,
$(s^--t^-, s-t, s^+-t^+)$ is constant on the domain of competition, which 
ensures fellow-traveling.

It follows that the area between the traces is linear on $U_\delta$,
as is the area between each of the shapes and its respective trace.
(Area between two planar paths with different endpoints is measured by
closing up with a straight chord.)  Clearly the boost of each shape is
also linear on $\DomComp$. Thus $f_{\omega\omega'}$ is linear for
each value of $\delta$ in the case where $\omega$ and $\omega'$ have
the same combinatorial type.  (Notice that in the case where $\tau$
and $\tau'$ might a priori differ as in the second case of
Lemma~\ref{lem::comb-types}, because of our restriction to
$\Cone_0$, they actually fellow-travel and the argument goes
through.)

Next, consider the case where $\tau$ and $\tau'$ are almost the same
combinatorial type.  In this case $\omega(X)$ and $\omega'(X')$ are
also of almost the same combinatorial type.  For specificity let us
consider the case where $\II+1=\II'$, $\JJ=\JJ'$, and $s^-$ and
$t-t^-$ are both bounded, so that there are only finitely many
possible pairs $(s^-,t-t^-)$.  For any such pair, the subset of
$U_\delta$ realizing that pair is defined by linear equations.  If we
fix those values---i.e., treat $a_\II^{s^-}a_{\II+1}^{t-t^-}$ as a
break word in $\omega(X)$---we can define new traces of the same
combinatorial type and appeal to the case above.

Finally, we turn to the case where $\tau$ and $\tau'$ end close to the 
origin and have different  types.  Here, $\omega$
and $\omega'$ can only compete when $\tau$ and $\tau'$ are close to 
being the full polygon.   But this implies that there are finitely many
values $(a,b)$ for which they
compete.  Furthermore, the set of $(X,X')$ mapping to each of these
finitely many $(a,b)$ is determined by linear equalities and inequalities.  For
each such $(a,b)$, the areas of $\omega(X)$ and $\omega(X')$ differ from
a full isoperimetrix of scale $s$ by amounts which are linear in $X$ and $X'$
respectively.  Thus their areas differ from each other  by amounts which are 
linear in $X$ and $X'$, and once again their respective boosts are also linear
in $X$ and $X'$.  The result now follows.
\end{proof}

\subsection{Testing geodesity for shapes}

Consider the set
$$ \{(a,b,w_n+j) : n\ge n_0(a,b), \quad 1 \le j \le w_{n+1}-w_n\},$$ 
containing the elements of the Heisenberg group in the (positive)
regular range, i.e., the set of $(a,b,c)\in H(\Z)$ with
$c>W=w_{n_0}(a,b)$.  By Bounded Difference
(Cor~\ref{cor::boundedLengthDifference}), such an element
$(a,b,w_n+j)$ has word length between $n+1$ and $n+K$.

Since $\Shape_K$ is a finite set, we can fix an arbitrary ordering of
its elements.

\begin{definition}
For a general shape $\omega$, let 
 $G^\Delta_\omega(n)$ be the set of $(a,b,j)\in \Z^3$ such that $1\le j \le w_{n+1}-w_n$
and $\omega$ is the first shape to geodesically realize $(a,b,w_n+j)$ at length $n+\Delta$.
  \end{definition}

 \begin{Thm}[Deciding geodesity for shapes] For
  each shape $\omega$ and each $0\le \Delta \le K$,
the  $G^\Delta_\omega(n)$ form  a bounded rational family in $\Z^3$.
\end{Thm}

\begin{proof}
We will show that membership in $G_\omega^\Delta(n)$ is tested
by finitely many linear equations, linear inequalities, and congruences.  

For a shape $\omega$, consider
$$ A_\omega(n) = \{(a,b) : \text{$\omega$ produces a spelling of
  length $n$ over $(a,b)$} \}.$$
To see that $A_\omega(n)$ is a rational family, recall that $X\mapsto \ell(\omega(X))=n$ is an
affine map.  Thus the sets $\{X \in \Dom_K(\omega) :
\ell(\omega(X))=n\}$ constitute a rational family.  The map $X\mapsto
(a,b)$ is also affine and thus the sets $A_\omega(n)$ are the affine
push-forwards of a rational family and hence themselves rational. 

For each  shape $\omega$ consider
$$H_\omega(n) = \{(a,b) : \text{$\omega$ realizes the highest-height
  element $(a,b,w_n)$ at length $n$} \},$$ which is empty unless
$\omega$ is a simple shape.  We claim that for each simple shape
$\omega$, $H_\omega(n)$ is a rational family.  For each $(a,b) \in
A_\omega(n)$, $\omega$ fails to produce the highest-height element if
 there is $\omega'$ producing a higher element over $(a,b)$ at
length at most $n$.  However, since $w_{n-2} < w_{n}$ this only needs
to be tested for length $n$ and $n-1$. Thus, for each potential
competitor $\omega'$ only two inequalities need to be tested.  But
these are tested by the linear inequality $f_{\omega,\omega'}(X,X')\ge
0$ at $\delta=0$ and $\delta=-1$.  It follows that the sets
$H_\omega(n)$ form a rational family as claimed.

Now for each pair of shapes $\alpha$ and $\beta$, note that 
$$
  H_\alpha(n+1) \cap H_\beta(n)= \{(a,b) : \text{$\beta$
    realizes $w_n(a,b)$ and 
   $\alpha$ realizes $w_{n+1}(a,b)$} \}.
$$
This is a rational family picking out positions at which $\alpha$ is highest-height at length $n+1$
and $\beta$ is highest-height at length $n$.
Given $(a,b)$, we can search the finite list of shapes to find such a pair, and then 
$(a,b,n)$ affinely determine $\S_\alpha$ and $(a,b,n+1)$ determine $\S_\beta$ so that 
$\alpha(\S_\alpha)$ and $\beta(\S_\beta)$ are the highest-height paths.
Thus, we can
test the requirement that $j$ satisfy $1 \le j \le w_{n+1}(a,b) - w_n(a,b)$ using  equations which are linear in our data
by seeing whether there exist shapes $\alpha,\beta$ for which $j\le f_{\alpha\beta}(\S_\alpha,\S_\beta)$ at $\delta=1$.

The requirement that $\omega$ realizes $(a,b,w_n+j)$ at length
$n+\Delta$ is similarly tested by $j=f_{\omega\beta}(\S,\S_\beta)$ 
at $\delta=\Delta$, i.e., by linear
equalities and inequalities.

Finally, for any $\omega$ which realizes $(a,b,w_n+j)$ at length
$n+\Delta$, we must test whether this is geodesic, i.e., whether this length is shortest-possible.  
This is accomplished
by testing all potential competitors $\omega'$ at lengths $\Delta' <
\Delta$.  This is finitely many competitors $\omega'$ and finitely many
values $\Delta'$, and therefore determined by finitely many linear
equalities and inequalities.  

Finally, to see that $\omega$ is the lowest-numbered shape to produce such a geodesic,
we simply check $\omega'<\omega$ at length $\Delta$.  
\end{proof}

\subsection{Linear comparison for patterns}

We will establish linear competition for patterns as we
did for shapes above.  For patterns $\W$ and $\W'$, define
$$\DomComp(\W,\W') = \{(n_1,n_2,n_1',n_2'): (a,b)=(a',b')\},$$
requiring that both paths end at the same horizontal position.
Notice that on $\DomComp(\W,\W')$, 
the length difference $\ell(\W(n_1,n_2))
- \ell(\W'(n_1',n_2'))$ is constant.  

\begin{lemma}[Linear comparison for patterns]
The comparison functions $\max\III_\W - \max \III_{\W'}$
and $\min\III_\W - \min \III_{\W'}$ are affine on 
a finite partition of $\DomComp(\W,\W')$.

Equivalently, these can be regarded as  affine functions on 
$(n_1,n_2)$, or  affine in $(a,b)$ on those
$(a,b)$ whose fibers are reached by both $\W$ and $\W'$.
\end{lemma}

\begin{proof}
The three statements are equivalent because on the appropriate sets,
each of the three quantities, $(n_1,n_2,n_1',n_2')$, $(n_1,n_2)$ and
$(a,b)$ determines the other two by an affine map.
Linearity follows from the fact that the tops of the intervals are
given piecewise by quadratic polynomials with the same leading coefficient,
and the bottoms of the intervals are piecewise linear, over finitely many
rational families that partition $\DomComp$.
\end{proof}

\subsection{Testing geodesity for patterns}

Each pattern $\W$ is easily seen to 
determine maps from $(n_1,n_2)$ to length, horizontal position,
and $\ell(\W(n_1,n_2)) - n_0$.  Notice that for each
$\W$,  $(a,b,\ell)$ is an affine function of $(n_1,n_2)$, and
the map $(n_1,n_2) \mapsto (a,b)$ is injective.

Since $\Patt_K$ is a finite set, we can fix an arbitrary ordering of patterns as we did for shapes.

\begin{definition}
For a  pattern $\W$, let 
$G^\Delta_\W(n)$ be the set of $(a,b)\in \Z^2$ such that $n=n_0(a,b)=|(a,b)|_{\pi(S)}$ and 
$\W$  realizes some  $(a,b,c)$ at length $n+\Delta$.
  \end{definition}

 \begin{lemma}[Positions reached by patterns] For
  each shape $\W$ and each $0\le \Delta \le K$,
the  $G^\Delta_\W(n)$ form  a bounded rational family in $\Z^2$.
\end{lemma}

\begin{proof}
The set of $(a,b)$ reached by $\W$ is the push-forward under an affine map of the set of
  non-negative pairs $(n_1,n_2)$.
Now observe that in the $\II$th sector of the plane, the length $n_0=n_0(a,b)$ is a periodic linear
function in which the linear coefficient is independent of $(a,b)$ and
the constant term depends on the congruence class of $(a,b)$ modulo
the group generated by $\A_\II$ and $\A_{\II+1}$.  Note also that if
$\W(n_1,n_2)$ ends over $(a,b)$, then $(n_1,n_2)$ and $(a,b)$ are
affine functions of each other. Of course then length of
$\W(n_1,n_2)$ is an affine function of $(n_1,n_2)$.  Thus the difference
$\ell(\W(n_1,n_2))-n$, which gives $\Delta$, is 
is a periodic function, and the  result follows.
\end{proof}

\begin{corollary}[Counting with patterns]
\label{cor::JOmegaIsPolynomial}
For each $\W$  and $0\le \Delta \le K$ there are  polynomials 
$p_{\W}^\Delta(a,b)$ of degree at most two such
that for $(a,b) \in G_\W^\Delta(n)$ the number of group elements $(a,b,c)$ with $c\ge 0$ 
geodesically spelled by $\W$ at length $n+\Delta$, and by no smaller-numbered pattern, is given by 
$p_{\W}^\Delta(a,b)$.
\end{corollary}

\begin{proof}
Clearly, the unstable elements of length $n_0+\Delta$ are those
reached by some pattern $\W$ at length $n_0+\Delta$ but not
by $\W'$ with length $n_0+\Delta'$ for any $\Delta' < \Delta$.

The $p_\W$ are defined by making the  comparisons of the interval $\III_\W$ 
against competing intervals $\III_{\W'}$, and enumerating the points over $(a,b)$ assigned to $\W$
as a finite sum/difference of the appropriate quadratic polynomials.
\end{proof}

\section{The growth series}

The growth series of $(H,S)$ is now given as 
follows.  The generators $S$ determine a constant $K$ so that
the positive-height regular elements are enumerated by 
$$ \SSS^{\sf reg}(x)=
\sum_{\omega} \quad
\sum_{n=0}^\infty  \quad
 \sum_{\Delta=0}^K \quad 
\sum_{G_\omega^\Delta(n)}
 x^{\Delta}\, x^n ,$$ 
where $\omega\in \Shape_K$ are the shapes described above.

The series enumerating unstable elements with $c\ge 0$ is
$$ \SSS^{\sf uns}(x)=
\sum_{\W} \quad
\sum_{n=0}^\infty  \quad
 \sum_{\Delta=0}^K \quad 
\sum_{G_\W^\Delta(n)}
 p_\W^\Delta(a,b) \, x^\Delta \, x^n 
,$$ 
where $\W\in\Patt_K$ are the patterns described above.
The difference in appearance between the two expressions corresponds
to the fact that regular \CC geodesics of a certain length only hit each fiber
in a single point, while unstable \CC geodesics may hit in an interval 
of size that is quadratic in the length.

Both series are rational by Theorem~\ref{thm::rational-counting}, because
$\Shape_K$ and $\Patt_K$ are finite sets, the $G(n)$ are bounded
rational families, and the $p$ are polynomial.
We then appeal to the height-reversing bijection $g\mapsto g^{-1}$ to similarly 
count the elements of non-positive height.  This double-counts the 
elements at height zero.

\begin{lemma}[Zero-height elements]\label{lem::zeroHeight}
Let $\sigma^\zero(n)=\#\{(a,b,0) : |(a,b,0)|_S=n\}$ be the spherical growth function of height-zero
elements.  Then
$\SSS^{\zero}(x)=\sum \sigma^\zero(n) x^n$ is rational.
\end{lemma}

\begin{proof}
The fiber over $(a,b)$ has an element with $c=0$ if and only if $ab$
is even.  Thus, our problem reduces to counting the set of such $(a,b)
\in \Z^2$ with respect to the generating set $\pi(S)$.  It is
well-known that the set of lex-least geodesics in an abelian group is
a regular language.  Those ending at an element $(a,b)$ with $ab$ even
is a regular subset of these.  The set in question therefore has
rational growth.
\end{proof}

Finally, we have 
$$\SSS(x)=2\sdot\SSS^{\sf reg}(x) + 2\sdot\SSS^{\sf uns}(x)-\SSS^{\zero}(x).$$
This establishes that 
the spherical growth series $\SSS(x)$ and thus also the growth series $\BBB(x)$  is rational 
for any finite generating set of $H(\Z)$, finishing Theorem~\ref{thm::rationalGrowth}.

\section{Applications, remarks, and questions}

\subsection{Languages}

Each shape defines a language $\calL(\omega)$. For $\JJ > \II+1$,
these languages are not regular.  For $\JJ > \II+2$, they are not
context-free. 

This is attributable to non-commutativity: what could be
accomplished with a bounded counter if the group were abelian is a
non-regular language otherwise.  For instance, $\{a^nb^n\}$ is
non-regular, even though $\{(ab)^*\}$ enumerates words with the same
letters.  The words represented by our shapes of geodesics need to be nearly
balanced, and this breaks regularity.

It was pointed out
to us by Cyril Banderier that a recursion with positive integer
coefficients implies the existence of {\em some} regular language enumerated
by the function, though not necessarily the language of geodesics for
$(G,S)$.  This holds in the special case of $(H,\std)$, which is
extremely intriguing.

\subsection{Cone types}

We recall the definition of {\em cone type} from \cite{cannon1984}.

\begin{definition}
Consider the Cayley graph $\Cay(G,S)$ of group $g$ with generating set
$S$.  Given $g \in G$, the {\em cone at $G$}, denoted $C(g)$, consists of all paths
$\sigma$ based at $g$ with the property that  word length $|\sigma(t)|$ 
is strictly increasing along
$\sigma$.  The {\em cone type} of $g$ consists of the cone of $g$
translated to the origin, i.e., $g^{-1}(C(g))$.
\end{definition}

For $\Cay(G,S)$ to have finitely many cone types is almost exactly the
same thing as having the language of geodesics in $\Cay(G,S)$ be a
regular language.  If $\Cay(G,S)$ has finitely many cone types, these
cone types can be used as the states of a finite state automaton which
accepts the language of geodesics.  This is because the cone type of
$G$ tells us exactly which generators are outbound at $g$.  However,
the cone type of $g$ encodes additional information, namely which
edges are ``half outbound":  if an edge $e$ of $\Cay(G,S)$ connects
two elements $g$ and $g'$ with $|g|=|g'|=n$, then the midpoint of this
edge is at distance $n+\frac 1 2$ from the origin.  
We believe that there is no known example
of a group presentation for which 
the language of geodesics is regular, but which has
infinitely many cone types.

From the shape theorem we easily recover the (already known) fact that 
$H$ has infinitely many cone types in every generating set.  In
  particular, it has no generating set where the language of geodesics
  is regular.

To see this, just note that there are infinitely many possibilities for how long a 
geodesic continues in a particular significant direction before turning to the 
successive direction, depending on what shape has reached the point
$g=(a,b,c)$ at what scale.


Brian Rushton has pointed out to us that the presence of infinitely many 
cone types implies that there is no associated subdivision rule.
(See \cite{rushton}.)

\subsection{Almost convexity}

A metric space is called {\em almost convex} $(k)$ or $AC(k)$
if there exists a constant $N(k)$
 such that for any two elements $x,y$ 
 in a common metric sphere $S_n(x_0)$ with $d(x,y)\le k$,
 there is 
a path of length at most $N$ connecting $x$ and $y$ in $B_n(x_0)$.  
That is, convexity would require that for two points on a sphere, 
connecting them inside the ball is efficient; 
almost-convexity is the existence of an additive bound on the inefficiency.
This was defined by Cannon in \cite{cannon1987}, where he showed
that for Cayley graphs of finitely generated groups, $AC(2)\implies 
AC(k) \quad \forall k$.  The importance of this property is that it gives 
a fast algorithm for constructing the Cayley graph.  
Almost-convexity is known for hyperbolic groups and 
virtually abelian groups with any finite set of generators, and for Coxeter groups and certain 3--manifold groups
with standard generators.
Several weakenings and strengthenings of the property
have been proposed and studied by various authors.  It was established for $H(\Z)$ 
with standard generators in \cite{shapiro1989}, but to our knowledge
has not been extended to arbitrary generators, which we settle here
by using once again the comparison of the \CC and word metrics.

Intriguingly, the dissertation of Carsten Thiel \cite{thiel} establishes that higher Heisenberg groups
are {\em not} AC in their standard generators, which corresponds remarkably to 
Stoll's finding of non-rational growth for the same examples.  

\begin{lemma} 
The \CC metric on $H(\R)$ induced by any rational polygonal norm is almost convex.
\end{lemma}

\begin{proof}
Consider $x,y\in \ccsphere_n$ with $\dCC(x,y)\le 2$.
First we show that if there exist geodesics $\overline{\zero x}$ and $\overline{\zero y}$ that $K$--fellow-travel in projection, 
then there exists a connecting path $x\to y$ of bounded length inside the ball. 
To construct this path, begin with a constant $m\gg 1$.  We will build a path from $\pi(x)$ to $\pi(y)$ as follows:  backtrack distance $mK$ along 
 $\pi(\overline{\zero x})$.
Connect geodesically to the point $w$ that is $n-mK$ from the origin along $\pi(\overline{\zero y})$ and finish by connecting $w$ to $\pi(y)$ 
along $\pi(\overline{\zero y})$.  This path has length at most $(2m+1)K$.  Its lift connects $x$ not to $y$ but to something else in the same fiber over $\pi(y)$, 
differing in height by at most $mK^2$ because that is the most area that can be contained in the ``rectangular" strip enclosed by the path we have built.
To correct this, we can splice a loop into our planar path at the point $w$.  This loop follows 
a parallelogram with sides $t\U$ and $t\V$ for some successive significant generators, where 
$t$ is chosen so that the area of the parallelogram, $t^2(\U\wedge\V)$,  is the height differential to be made up.  
This has length at most $4\sqrt{\frac{m}{\U\wedge\V}}K$.  Since $m$ was chosen to be large, this length is less than $mK$ and so the lift of the concatenated path stays inside $\ccball_n$.
Thus we have connected $x$ to $y$ by a path inside the ball, of length bounded independent of $x,y,n$.

To complete the proof, we must reduce to this case.
By possibly inserting one extra point $z$ and separately considering the two pairs $x,z$ and $z,y$, we will cover all possibilities with the following cases.

{\bf Case 1:}  $x,y$ both unstable and in the same sector.

Then there are fellow-traveling geodesics as required:  if the sector is between significant directions $\A_\II$ and $\A_{\II+1}$, then 
$x$ is reached in exactly one way by a geodesic whose shadow is of the form $\A_\II^i\A_{\II+1}^j\A_\II^k$.  Likewise $y$ has a unique such 
geodesic, and they must fellow-travel to reach nearby endpoints.

{\bf Case 2:} $x,y$ both regular and of the same combinatorial type.

In this case, the geodesics from the origin are unique, and both project to 
 $P$--arcs for the defining polygon $P$ of the norm with the same combinatorial type.
From the fellow-traveling lemma for $P$--arcs (Lemma~\ref{lem::comb-types}) we know that these fellow-travel in projection.

{\bf Case 3:} One of $x,y$ projects to the origin (say $y=(0,0,c)$).  

In this case we fix any geodesic from the origin to $x$.  There are many geodesics reaching $y$ (corresponding to choosing any starting position on $P$), 
and we can take one of the same combinatorial type as the path chosen for $x$.  These then fellow-travel in projection.
\end{proof}

\begin{theorem} The Heisenberg group is almost convex with any word metric. \end{theorem}

\begin{proof}
Start with $g_1,g_2$ with $|g_1|=|g_2|=n$ and $|g_1g_2^{-1}|\le 2$, and 
let $K$ be the constant bounding the difference between the word
and \CC metrics, as in Cor~\ref{cor::boundedLengthDifference}.  
Then if $B_n$ is the ball of radius $n$ in the word metric and 
$\ccball_n$ is the ball of radius $n$ in the associated \CC metric, 
we have $B_n \subset \ccball_{n+K}$ and 
$\ccball_{n-K}\cap H(\Z) \subseteq B_n$.

Fix any $p\gg 2K$.  Let $h_1$ be a group element obtained by backtracking
$p$ steps along a geodesic spelling of $g_1$, so that $|h_1|=n-p$, and 
define $h_2$ similarly.  The distance $|h_1h_2^{-1}|$ is at most 
$2+2p$, and since the (continuous) group is $AC(2+2p)$, there is a constant
$N(2+2p)$ so that a \CC path $\gamma$ exists between $h_1$ and $h_2$ 
of length at most $N$ and contained totally inside $\ccball_{n-p+K}$.
As $\gamma$ is traversed from $h_1$ to $h_2$, construct an ordered
set of integer points by choosing a nearest point at each time.  
Since the diameter of a fundamental domain for $H(\Z)$ is bounded,
say by $\Delta$,  
each of these points is contained in the $\Delta$--neighborhood of 
$\gamma$ and therefore each is within $2\Delta$ \CC distance of the previous
and next point in the sequence.  These round-off points all lie in 
$\ccball_{n-p+K+\Delta}$.  Two successive points can be connected by 
a word path of length at most $2\Delta+K$, and the word path 
from $h_1$ to $h_2$ built by concatenating these must lie inside 
$\ccball_{n-p+2K+2\Delta}$.  There are at most $N/2\Delta$
round-off points, so the total length of the word path from $h_1$ 
to $h_2$ is bounded by $(N/2\Delta)(2\Delta+K)$.
Since $p$ was chosen to ensure that $n-p+K+2\Delta<n-K$, 
this path lies inside $\ccball_{n-K}\cap H(\Z) \subseteq B_n$.
Piecing this together we obtain a path from $g_1$ to $g_2$ inside
$B_n$ of length at most $(N/2\Delta)(2\Delta+K)+2p$.
\end{proof}

\subsection{Open questions}

\subsubsection{Scope of rational growth in the nilpotent class}

Our argument should carry through with small modifications for groups that are virtually 
$H(\Z)\times \Z^d$.  
We know from Stoll's result that not all two-step groups have rational growth, even with 
respect to their standard generators.  However it is possible (for instance)
that free nilpotent groups do.

\begin{question} Which nilpotent groups have rational growth in all generating sets? \end{question}

On the other hand, one could try to mimic and extend the Stoll construction.

\begin{question} Does every nilpotent group have rational growth with respect to at least one
generating set?  In the other direction, for which nilpotent groups is the fundamental volume 
transcendental for standard 
generators (which would rule out rationality by Thm~\ref{thm::stoll})?
\end{question}

\subsubsection{Period and coefficients}

In the polynomial range (i.e., $f(n)\le An^d$ for some $A,d$), rational growth is equivalent to 
the property that $f(n)$  is  {\em eventually quasi-polynomial}, i.e., 
there are a finite period $N$,  polynomials $f_1,\dots,f_N$, and a threshold $T$ such that 
$$n\ge T, \quad n=kN+i \implies f(n)=f_i(n).$$
For example, Shapiro's computation of the spherical growth for the Heisenberg group with standard
generators
showed it to be eventually quasipolynomial of period twelve, and in fact only the constant term
oscillates:
$$\sigma(n)=\frac{1}{18}\left( 31n^3-57n^2+105n +c_n\right),$$
where $c_n=-7,-14,9,-16,-23,18,-7,32,9,2,-23,0$, and then repeats mod $12$, for $n\ge 1$.  
(So that $\sigma(1)=4$, $\sigma(2)=12$, and so on.)

It follows that the (ball) growth function $\beta(n)=\sum_{k=0}^n \sigma(k)$ is 
also quasipolynomial of period twelve,
with only its constant term oscillating.  We note that this implies that the growth function
for standard generators is within bounded distance of a true polynomial in $n$.

Preliminary calculations indicate that several other generating sets
also have the property that only the constant terms oscillate; in these examples, the periods 
relate both to the sidedness of the fundamental polygon and to the index of the 
sublattice of $\Z^2$ generated by its extreme points.  

\begin{question} How does the generating set $S$ determine the period of quasipolynomiality
of the growth function?  
Which coefficients oscillate?  We know that the top coefficient of $\beta(n)$ 
is the volume of the \CC ball; is the second 
coefficient well-defined, and if so is it a ``surface area"?  
Are all growth functions bounded distance from polynomials?
\end{question}


\begin{thebibliography}{99}

\bibitem{bass1972} H.\ Bass,
The degree of polynomial growth of finitely generated nilpotent groups. Proc. London Math. Soc. (3) 25 (1972), 603--614.


\bibitem{benson1983} 
M.\ Benson, Growth series of finite extensions of $\Z^n$ are rational, Invent. Math. 73 (1983), no. 2, 251--269. 

\bibitem{benson1987} Max Benson, On the rational growth of virtually nilpotent groups, Ann. Math. Stud 111 (1987), 185--196.


\bibitem{breuillard} E.\ Breuillard, Geometry of groups of polynomial growth and shape of large balls.\\
{\sf arXiv:0704.0095}

\bibitem{breuillard-ledonne} E.\ Breuillard and E.\ LeDonne, 
On the rate of convergence to the asymptotic cone for nilpotent groups and subFinsler geometry. Proc. Natl. Acad. Sci. USA 110 (2013), no. 48, 19220--19226.

\bibitem{busemann} H.\ Busemann, The isoperimetric problem in the {M}inkowski plane.
AJM 69 (1947), 863--871.

\bibitem{cannon1980}
J.\ Cannon, 
The growth of the closed surface groups and compact hyperbolic Coxeter groups. Circulated typescript, Univ. Wisconsin, 1980.

\bibitem{cannon1984} J.\ Cannon, 
The combinatorial structure of cocompact discrete hyperbolic groups. 
Geom. Dedicata 16 (1984), no. 2, 123--148. 

\bibitem{cannon1987} J.\ Cannon, Almost convex groups, Geom. Dedicata
  22 (1987), no. 2, 197--210.


\bibitem{capogna} L.\ Capogna, D.\ Danielli, S.\ Pauls and J.\ Tyson,
{\em An Introduction to the Heisenberg Group and to the Sub-Riemannian Isoperimetric Problem}. 
Birkhauser, Progress in Mathematics, 2007.

\bibitem {dm} M.\ Duchin and C.P.\ Mooney, 
{\em Fine asymptotic geometry of the Heisenberg group},
Indiana University Math Journal 63 No. 3 (2014), 885--916.

\bibitem{wordproc}
D.B.A.\ Epstein, J.W.\ Cannon, D.F.\ Holt, S.V.F.\ Levy, M.S.\ Paterson, and W.P.\ Thurston, 
Word processing in groups. Jones and Bartlett, 1992.

\bibitem{grig-dlh} R.\ Grigorchuk and 
P.\ de la Harpe, On problems related to growth, entropy, and spectrum in group theory, J. Dynam. Control Systems 3 (1997), no. 1, 51--89. 

\bibitem{gromov-poly-growth} M.\ Gromov,
Groups of polynomial growth and expanding maps. 
Inst. Hautes �tudes Sci. Publ. Math. No. 53 (1981), 53--73. 

\bibitem{gromov-hyp-gps} M.\ Gromov,
Hyperbolic groups. Essays in group theory, 75--263, 
Math. Sci. Res. Inst. Publ., 8, Springer, New York, 1987. 

\bibitem{guivarch1970} Y.\ Guivarc'h, Groupes de Lie \`a croissance polynomiale. (French) 
C. R. Acad. Sci. Paris S\'er. A-B 271 1970 A237--A239. 

\bibitem{guivarch1973} Y.\ Guivarc'h, Croissance polynomiale et p�riodes des fonctions harmoniques. (French) 
Bull. Soc. Math. France 101 (1973), 333--379. 

\bibitem{dlh} P.\ de la Harpe, Topics in geometric group theory. Chicago Lectures in Mathematics. University of Chicago Press, Chicago, IL, 2000.

\bibitem{krat} S.A.\ Krat, Asymptotic properties of the Heisenberg group.
Journal of Mathematical Sciences, Vol. 110, No. 4 (2002) 2824--2840.

\bibitem{mann} A.\ Mann, How groups grow. London Mathematical Society Lecture Note Series, 395. Cambridge University Press, Cambridge, 2012.

\bibitem{neumann-shapiro} W.\ Neumann and M.\ Shapiro, 
Automatic structures, rational growth, and geometrically finite hyperbolic groups. Invent. Math. 120 (1995), no. 2, 
259--287.

\bibitem{pansu} P.\ Pansu, Croissance des boules et des g\'eod\'esiques ferm\'ees dans les nilvari\'et\'es.
 Ergodic Theory Dynam. Systems  3  (1983),  no. 3, 415--445. 

\bibitem{rushton} B.\ Rushton, 
Classification of subdivision rules for geometric groups of low dimension. 
Conform. Geom. Dyn. 18 (2014), 171--191. 


\bibitem{shapiro1989} M.\ Shapiro, A geometric approach to the almost convexity and growth of some nilpotent groups.  Math. Ann. 285, 601--624 (1989).

\bibitem{stoll1996} M.\ Stoll,
Rational and transcendental growth series for the higher Heisenberg groups.
Invent. math. 126, 85--109 (1996).

\bibitem{stoll1998} M.\ Stoll,
On the asymptotics of the growth of 2--step nilpotent groups.
J. London Math. Soc. (2) 58 (1998) 38--48.

\bibitem{thiel} C.\ Thiel,
Zur fast-Konvexit\"at einiger nilpotenter Gruppen. (German) [On the almost convexity of some nilpotent groups] 
Dissertation, Rheinische Friedrich-Wilhelms-Universit\"at Bonn, Bonn, 1991.

\end{thebibliography}

\end{document}